\documentclass[a4paper,11pt,reqno]{amsart}

\usepackage[T1]{fontenc}
\usepackage{lmodern}
\usepackage{amsmath,amssymb,amsthm}

\usepackage[top=1.5in,bottom=1.3in,left=1.3in,right=1.3in,marginparwidth=1.5cm]{geometry}
\usepackage{amsrefs}
\usepackage[colorlinks=true, citecolor=blue, urlcolor=blue, linkcolor=blue
]{hyperref}

\usepackage{array}
\usepackage{booktabs}

\allowdisplaybreaks
\numberwithin{equation}{section}

\newtheorem{theorem}{Theorem}[section]
\newtheorem{proposition}[theorem]{Proposition}
\newtheorem{lemma}[theorem]{Lemma}
\newtheorem{corollary}[theorem]{Corollary}
\theoremstyle{remark}

\newcommand{\fF}{F_4}
\newcommand{\eE}{E_6}

\newcommand{\g}{\mathfrak g}
 
\newcommand{\h}{\mathfrak h}
\newcommand{\n}{\mathfrak n}
\newcommand{\vac}{\mathbf 1}
\newcommand{\OO}{\mathbb O}
\newcommand{\C}{\mathbb C}
\newcommand{\D}{\mathcal D}
\newcommand{\cH}{\mathcal H}

\newcommand{\HP}{\operatorname{HP}_0}
\newcommand{\ad}{\operatorname{ad}}
\newcommand{\Spec}{\operatorname{Spec}}

\newcommand{\HC}{\operatorname{HC}}
\newcommand{\gr}{\operatorname{gr}}
\newcommand{\tr}{\operatorname{tr}}

\title[Quasi-lisse affine vertex algebras $L_{-1-h/6}(\g)$]{On quasi-lisse affine vertex algebras $L_{-1-h/6}(\g)$ from rank-one Argyres--Douglas theories}

\author{Dra\v{z}en Adamovi\'{c}}
\email{adamovic@math.hr}
\address{Department of Mathematics, Faculty of Science, University of Zagreb, Bijeni\v{c}ka 30, 10 000 Zagreb, Croatia}

\author{Xuanzhong Dai} 
\email{xzdai@nju.edu.cn}
\address{School of Mathematics, Nanjing University,
Nanjing 210093, China}

\author{Kaiwen Sun}
\email{kwsun@ustc.edu.cn}

\address{School of Mathematical Sciences, University of Science and Technology of China, Hefei 230026, Anhui, China}

\date{}

\begin{document}

\begin{abstract}
We study certain simple affine vertex algebras $L_{-1-h/6}(\g)$ from rank-one Argyres--Douglas theories, where $\g$ is a non-simply-laced simple
Lie algebra and $h$ is its Coxeter number.
For $L_{-3}(\fF)$, we determine the maximal ideal of
$V^{-3}(\fF)$, Zhu's algebra, and the associated variety.
We also classify the irreducible highest-weight
$L_{-3}(\fF)$-modules and determine the characters of its
irreducible ordinary modules.
The associated variety is $\overline{\OO}_{\widetilde A_1}$,
and the level $-3$ is totally collapsing for the orbit
$\OO_{\widetilde A_1}$.
The Poisson coinvariants of the $C_2$-algebra of
$L_{-3}(\fF)$ are isomorphic to $\C[\ell]/(\ell^5)$,
where $\ell$ is induced by the conformal vector.
Consequently, the characters of ordinary
$L_{-3}(\fF)$-modules satisfy a fifth-order modular linear
differential equation.
We also give modular linear differential equations of
orders five and eleven satisfied by the characters of
ordinary modules for $L_{-2}(B_3)$ and $L_{-2}(G_2)$ respectively.
\end{abstract}

\maketitle

\section{Introduction}


The correspondence between four-dimensional $\mathcal{N}=2$
superconformal field theories (SCFTs) and vertex operator algebras has revealed remarkable connections between supersymmetric quantum field theory, representation theory, and Poisson geometry. The construction of Beem, Lemos, Liendo, Peelaers, Rastelli, and van Rees
\cite{BeemEtAl} associates a vertex operator algebra to the Schur
sector of any such four-dimensional SCFT. Its operator product
expansions encode Schur correlation functions, while its vacuum
character recovers the Schur index. 
Consequently, questions about Schur operators and their relations
can be studied through the structure and representation theory of
vertex algebras. This is particularly useful for strongly coupled
theories, for which a Lagrangian description is often unavailable.

An important feature of this correspondence is that unitary
four-dimensional theories generally give rise to non-unitary
two-dimensional vertex operator algebras. In the standard normalization, the central charges and affine levels in two and four dimensions are related by
$$   c_{\mathrm{2d}}=-12c_{\mathrm{4d}},
    \qquad
    k_{\mathrm{2d}}=-\frac{1}{2}k_{\mathrm{4d}}.$$
Thus affine vertex algebras at negative levels arise naturally.
Many physically relevant examples occur at non-admissible levels,
where the structure of the simple quotient and its representations
is considerably less understood than in the admissible case. 
For instance, at admissible levels, Arakawa \cite{ArakawaAdmissible} has established the rationality of simple affine vertex algebras in category $\mathcal O$. While at non-admissible levels, there is no general results for such affine vertex algebras. Determining their defining relations, associated varieties, and
characters therefore provides both new examples in representation
theory and concrete information about protected sectors of
four-dimensional theories.

Quasi-lisse vertex algebras were introduced by Arakawa and Kawasetsu in \cite{ArakawaKawasetsu}. A vertex algebra is called quasi-lisse if its
associated variety has finitely many symplectic leaves.
Arakawa and Moreau \cite{ArakawaMoreau} proved that several simple
affine vertex algebras at non-admissible levels have associated
varieties equal to closures of minimal nilpotent orbits.
These include the $L_{-2}(D_4)$ corresponding to the simplest 4D $\mathcal{N}=2$ SCFT $SU(2)$ gauge theory with four fundamentals as well as the $L_{-3}(E_6)$ corresponding to the $E_6$ Minahan--Nemeschansky theory. Further non-admissible examples whose associated varieties are
contained in the nilpotent cone were obtained in
\cite{ArakawaMoreauSheets}.

In this article, we mainly study the simple affine vertex algebra
$L_{-3}(\fF)$. Its physical significance arises from the study of
rank-one 4D $\mathcal{N}=2$ SCFTs, where the rank
is the complex dimension of the Coulomb branch. In the investigation
of low-rank theories, Li, Xie, and Yan \cite{LiXieYan} identified
$L_{-3}(\fF)$ as the vertex algebra associated with a rank-one Argyres--Douglas theory
with flavour symmetry of type $\fF$. This example belongs to a family
of affine vertex algebras of simple Lie algebras
$$
A_1 \subset A_2 \subset G_2 \subset B_3 \subset D_4 \subset F_4 \subset E_6 \subset E_7 \subset E_8
$$
at levels
$$
    k=-\frac{h}{6}-1,
$$
where $h$ is the Coxeter number of the corresponding simple Lie
algebra. 
This series of simple Lie algebras differs slightly from the Cvitanovi\'c--Deligne exceptional series \cites{Cvitanovic:1980ed,Deligne}, as do the levels considered here from those studied in \cite{ArakawaKawasetsu}.
For $\fF$, this gives the level $k=-3$, which is non-admissible.
The Sugawara central charge is $c_{\mathrm{2d}}=-26$, corresponding
to $c_{\mathrm{4d}}=13/6$. The study of $L_{-3}(\fF)$ offers
an exceptional, non-simply-laced example in which the predictions
of the four-dimensional theory can be compared with explicit
vertex-algebraic calculations. 

A second fundamental aspect of the 4D/2D correspondence concerns the geometry of the Higgs branch. Beem and Rastelli \cite{BeemRastelli} proposed that the Higgs branch of a four-dimensional theory is isomorphic
to the associated variety of its vertex algebra. Recall that the
associated variety $X_V$ is the reduced affine Poisson variety
defined by the $C_2$-algebra
$R_V=V/C_2(V)$. The proposed identification relates the geometry
of supersymmetric vacua to the algebraic relations among protected
operators. For $L_{-3}(\fF)$, the physical prediction of
\cite{LiXieYan} is
$$
    X_{L_{-3}(\fF)}
    =\overline{\OO}_{\widetilde A_1},
$$
where $\OO_{\widetilde A_1}$ is the nilpotent orbit of
complex dimension $22$ in $\fF$. In particular, this geometry
goes beyond the minimal nilpotent orbit closures that occur in
several familiar rank-one examples.

The conformal embedding
$$L_{-3}(\fF)\hookrightarrow L_{-3}(E_6),$$ established by
Adamovi\'c and Per\v{s}e \cite{AdamovicPerse}, provides another motivation for studying this algebra.
It is an exceptional counterpart of
the renowned conformal embedding chain of quasi-lisse vertex algebras with non-admissible levels \cite{AdamovicPerse}
$$
    L_{-2}(G_2)\hookrightarrow L_{-2}(B_3)
    \hookrightarrow L_{-2}(D_4).
$$
The decomposition of these conformal embeddings was determined in \cite{AdamovicPerse}. 
The $L_{-2}(G_2)$ and $L_{-2}(B_3)$ vertex algebras were recently studied in \cite{ADFLM}, where their
representations and associated varieties were related to
predictions from rank-one Argyres--Douglas theories and to questions about
orbifold vertex algebras. The case of $L_{-3}(\fF)$ remained
beyond the methods used there. Understanding this case is
therefore a natural next step in the study of non-admissible
affine vertex algebras motivated by four-dimensional physics.

A level $k$ is called collapsing if the simple $W$-algebra
$\mathcal W_k(\mathfrak g,f)$ is generated by its conformal-weight
one subspace. It is called totally collapsing if
$\mathcal W_k(\mathfrak g,f)\cong\mathbb C\mathbf1$.
Applications of collapsing levels to the representation theory
of affine vertex algebras were developed in \cite{AKMPP}.
Further results on collapsing levels at admissible and
non-admissible levels were obtained in
\cite{ArakawaVanEkerenMoreau} and \cite{WExtensions}, respectively.

Another non-admissible quasi-lisse affine vertex algebra is
$L_{-14/3}(D_4)$, recently studied in 
\cite{AdamovicVukorepa}. Its associated variety is the closure
of the subregular nilpotent orbit in $D_4$. Moreover, its
quantum Drinfeld--Sokolov reduction with respect to a
subregular nilpotent element is $\mathbb{C}\mathbf{1}$.

In this paper, we determine the maximal ideal of the universal affine vertex algebra $V^{-3}(\fF)$ and its image in the Zhu algebra, and classify the irreducible $L_{-3}(\fF)$-modules in category $\mathcal{O}$. These results determine the defining relations of the simple vertex algebra and the conditions they impose on its highest-weight modules. We also prove
the physics prediction for its associated variety. In particular,
$L_{-3}(\fF)$ is quasi-lisse, since its associated variety has
only finitely many symplectic leaves. Moreover, the level $-3$
is totally collapsing for $\OO_{\widetilde A_1}$, namely
the corresponding simple $W$-algebra is isomorphic to
$\mathbb{C}$. This gives a further connection between the
nilpotent geometry and quantum Drinfeld--Sokolov reduction.

The character theory provides a complementary link with physics. 
Arakawa and Kawasetsu \cite{ArakawaKawasetsu} proved that ordinary characters of quasi-lisse vertex algebras satisfy modular linear
differential equations. Through the 4D/2D
correspondence, such equations impose differential constraints
on Schur indices. Their significance is that an infinite
series counting protected states is governed by a finite-order equation. Beem and Rastelli \cite{BeemRastelli} further related this
phenomenon to null relations in the vertex algebra. This generalizes the connection between modular
differential equations and null-vector relations in the
vacuum module for 2D rational conformal field theories \cite{GaberdielKeller}.

In the $L_{-3}(\fF)$ case, we compute the Poisson coinvariants and obtain
$$
    \operatorname{HP}_0(R_V)
    :=R_V/\{R_V,R_V\}
    \cong \mathbb{C}[\ell]/(\ell^5),
    \qquad V=L_{-3}(\fF),
$$
where $\ell$ denotes the class induced by the conformal vector.
This calculation supplies information that is not visible from
the reduced associated variety alone. We determine the ordinary
characters and an explicit fifth-order modular linear
differential equation satisfied by them, and prove that its
order is minimal for the vacuum character. Thus the
Poisson-algebraic calculation is accompanied by a sharp
differential constraint on the corresponding Schur
index. We also obtain modular differential equations of orders
five and eleven for $L_{-2}(B_3)$ and $L_{-2}(G_2)$, respectively,
extending this analysis to the related non-simply-laced examples.

\subsection{Summary of results}

Set $\g=\fF$.  Let $\epsilon_1,\ldots,\epsilon_4$ be orthonormal, and we use
$$
 \alpha_1=\epsilon_2-\epsilon_3,\qquad
 \alpha_2=\epsilon_3-\epsilon_4,\qquad
 \alpha_3=\epsilon_4,\qquad
 \alpha_4=\frac12(\epsilon_1-\epsilon_2-\epsilon_3-\epsilon_4)
$$
and the normalization of root vectors from
\cite[formula (3.1)]{Perse}.  The fundamental weights are
\begin{align*}
 \omega_1&=\epsilon_1+\epsilon_2,&
 \omega_2&=2\epsilon_1+\epsilon_2+\epsilon_3,\\
 \omega_3&=\tfrac12(3\epsilon_1+\epsilon_2+\epsilon_3+\epsilon_4),&
 \omega_4&=\epsilon_1.
\end{align*}
Fix a triangular decomposition $\g=\n_-\oplus\h\oplus\n_+$ of $\g=\fF$. Put $h_i=\alpha_i^\vee$ and write $[e_\alpha,f_\alpha]=h_\alpha$. The invariant form $(\cdot\mid\cdot)$ is the normalized form, so $(e_\alpha\mid f_\alpha)=2/(\alpha\mid\alpha)$. In particular, $h^\vee=9$ and $\dim\g=52$.

Let $V^k(\g)$ be the universal affine
vertex algebra and $L_k(\g)$ its simple graded quotient.  
For a
vertex algebra $V$, write
$$
 R_V=V/C_2(V),\qquad X_V=\operatorname{Specm}R_V.
$$
We identify $R_{V^k(\g)}$ with $S(\g)$ by
$x(-1)\vac\mapsto x$.

\begin{theorem}\label{thm:main}
Let $\g=\fF$. The following statements hold.
\begin{enumerate}
\item
The maximal ideal of $V^{-3}(\g)$ is generated by a
$\widehat{\g}$-singular vector of conformal weight $3$
and $\g$-weight $\omega_2$.

\item
The associated variety is
$$
X_{L_{-3}(\g)}
=\overline{\OO}_{\widetilde A_1}
=\OO_{\widetilde A_1}\sqcup\OO_{A_1}\sqcup\{0\}.
$$
In particular, $L_{-3}(\g)$ is quasi-lisse and
$\dim X_{L_{-3}(\g)}=22$.

\item
For $f\in\OO_{\widetilde A_1}$, we have
$$
\mathcal W_{-3}(\g,f)\cong\C.
$$

\item
Up to isomorphism, there are $14$ irreducible
$L_{-3}(\g)$-modules in category $\mathcal O$.
Exactly two are ordinary:
$$
L_{-3}(\g),\qquad L_{-3}(\omega_4).
$$
\end{enumerate}
\end{theorem}

\begin{theorem}\label{thm:main-characters}
Let $V=L_{-3}(\fF)$, and let $\ell$ be the image of
the conformal vector in $R_V$.
\begin{enumerate}
\item
The space of Poisson coinvariants satisfies
$$
\operatorname{HP}_0(R_V)
=\frac{R_V}{\{R_V,R_V\}}
\cong\C[\ell]/(\ell^5)
$$
as a graded vector space.

\item
The normalized characters of all ordinary $V$-modules
satisfy the monic holomorphic
$\mathrm{SL}_2(\mathbb Z)$-MLDE of order five
given in Theorem~\ref{thm:mlde}.
The vacuum character satisfies no monic holomorphic
$\mathrm{SL}_2(\mathbb Z)$-MLDE of smaller order.
\end{enumerate}
\end{theorem}

Let $\sigma$ be the Dynkin diagram involution of $\eE$, whose
fixed-point subalgebra is $\fF$, and put
$\Gamma=\langle\sigma\rangle\cong\mathbb Z/2\mathbb Z$. By
\cite[Section~2.3.3]{FJLS23}, the projection
$$
p:\eE\longrightarrow\fF,
\qquad x\longmapsto \frac{x+\sigma(x)}{2},
$$
restricts to a finite, generically two-to-one morphism
$$
p:\overline{\OO}_{\min}^{E_6}
\longrightarrow
\overline{\OO}_{\widetilde A_1}
$$
and induces an isomorphism of Poisson varieties
$$
\overline{\OO}_{\min}^{E_6}/\Gamma
\cong
\overline{\OO}_{\widetilde A_1}.
$$
Hence the second result of Theorem \ref{thm:main} can be reformulated as follows
$$
X_{V^\Gamma}=X_V/\Gamma, \qquad
\text{ for } V=L_{-3}(\eE),
$$
where the associated variety of the $L_{-3}(\eE)$ is the minimal nilpotent orbit closure of $\eE$ by \cite{ArakawaMoreau}.  The analogous folding from $D_4$ to $G_2$ and $B_3$ was studied in \cite{ADFLM}.

The paper is organized as follows. In Section \ref{sec:prelim}, we first recall the necessary
background on affine vertex algebras and Zhu algebras. 
In Section \ref{sec:repn}, we
determine the maximal ideal of $V^{-3}(F_4)$ and describe the Zhu
algebra of its simple quotient, leading to the classification of
irreducible $L_{-3}(F_4)$-modules in category $\mathcal{O}$.
In Section \ref{sec:ass}, we determine the associated variety and establish the total
collapsing property for the orbit $\OO_{\widetilde A_1}$.
In Section \ref{sec:MLDE}, we subsequently compute the Poisson coinvariants of the
$C_2$-algebra and determine the ordinary characters and their
fifth-order modular linear differential equation. 
We conclude in Section \ref{sec:MLDE2}, where we obtain modular
differential equations of orders five and eleven for $L_{-2}(B_3)$
and $L_{-2}(G_2)$, respectively.

\section{Preliminaries and conventions}\label{sec:prelim}
\subsection{Modules for affine vertex algebras}

For $a\in\g$ and $m\in\mathbb Z$, write $a(m)=a\otimes t^m$. The Lie brackets in $\widehat{\g}$ are
\begin{equation*}
 [a(m),b(n)]=[a,b](m+n)+m\delta_{m+n,0}(a| b)K.
\end{equation*}
The vacuum module at level $k$ is
$$
 V^k(\g)=U(\widehat\g)\otimes_{U(\g[t]\oplus\C K)}\C_k,
$$
where $\g[t]$ acts trivially and $K$ acts by $k$. At $k\ne-h^\vee$, its Sugawara vector is
\begin{equation*}
 \omega=\frac1{2(k+h^\vee)}\sum_a u_a(-1)u^a(-1)\vac ,
\end{equation*}
for dual bases $\{u_a\}$ and $\{u^a\}$. The central charge is
\begin{equation*}
c(k)=\frac{k\dim\g}{k+h^\vee}
\end{equation*}
A highest-weight vector of a $\widehat{\g}$-module at level $k$, with $\g$-weight $\lambda$, has conformal weight
\begin{equation}\label{eq:2:conformal-weight}
 \Delta(\lambda)=\frac{(\lambda | \lambda+2\rho)}{2(k+h^\vee)}.
\end{equation}

We use category $\mathcal O$ for $\widehat{\g}$ with its standard Borel subalgebra; the degree derivation is fixed by $L_0$, so no additional degree shifts are counted. We denote by $L_k(\lambda)$ the irreducible highest-weight
$\widehat{\g}$-module at level $k$ whose lowest
conformal-weight subspace is the irreducible
$\g$-module $V_\g(\lambda)$. An ordinary $V^k(\g)$-module $M$ has a lower-bounded, semisimple $L_0$-grading with finite-dimensional eigenspaces. Its character is
$$
 \chi_M(\tau)=\tr_M q^{L_0-c/24},\qquad q=e^{2\pi i\tau}.
$$
An irreducible ordinary $V^k(\g)$-module has a finite-dimensional irreducible $\g$-module as its lowest eigenspace, and hence has dominant integral $\g$-highest weight.

\subsection{Zhu algebras, Poisson algebras and reduction}
Zhu's algebra of the universal affine vertex algebra is $U(\g)$ \cites{FrenkelZhu,Zhu}. Let $v$ be a $\widehat{\g}$-singular vector in $V^k(\g)$,
and let $J$ be the ideal of $V^k(\g)$ generated by $v$.
Under the identification $A(V^k(\g))\cong U(\g)$,
the image of $v$ generates the two-sided ideal defining
$A(V^k(\g)/J)$. Explicit singular-vector and Zhu-algebra calculations at negative integral levels in type $A$ give precedents for this approach \cite{APNonrational}. We shall use the unshifted Harish--Chandra projection
$$
 \HC:U(\g)_0\longrightarrow U(\h)=S(\h),
$$
obtained by projecting along $\n_-U(\g)+U(\g)\n_+$. In particular, $\HC(r)(\lambda)$ is the scalar by which a zero-weight element $r$ acts on a highest-weight vector.

For a vertex algebra $U$, the space $C_2(U)$ is spanned by $a_{(-2)}b$. The product and Poisson bracket on $R_U$ are
$$
 \bar a\,\bar b=\overline{a_{(-1)}b},\qquad
 \{\bar a,\bar b\}=\overline{a_{(0)}b}.
$$
For $U=V^k(\g)$, this identifies $R_U$ with $S(\g)$, with the Lie--Poisson bracket and with $a(-1)\vac$ mapping to $a$. This classical limit is the degree-zero part of the vertex Poisson algebra supplied by Li's filtration \cite{LiAbelianizing}. We distinguish the affine scheme $\Spec R_U$ from its reduced associated variety $X_U$ throughout the reduction argument.

For a nilpotent element $f$ and an $\mathfrak{sl}_2$-triple $(e,h,f)$, set $x=h/2$ and $\mathcal S_f=f+\g^e$. We write $H^0_{DS,f}$ for quantum Drinfeld--Sokolov reduction, and
$$
 \mathcal W^k(\g,f)=H^0_{DS,f}(V^k(\g)),\qquad
 \mathcal W_k(\g,f)
$$
for the universal $W$-algebra and its simple graded quotient. We use the BRST construction of generalized quantum Hamiltonian reduction in \cite{KacRoanWakimoto}. The exactness and associated-scheme theorems of \cite{ArakawaDS} apply to the $V^{-3}(\g)$-modules considered below. 

\section{Representations of $L_{-3}(\fF)$}\label{sec:repn}

The reducibility of the vacuum module $V^{-3}(\g)$ follows
from the general criterion of Gorelik and Kac
\cite{GorelikKac}.  In this section we construct a $\widehat{\g}$-singular vector of conformal
weight $3$ in $V^{-3}(\g)$ and prove that it generates
the maximal proper ideal of $V^{-3}(\g)$.
Using Zhu's algebra, we classify the irreducible
$L_{-3}(\g)$-modules in category $\mathcal O$.
Exactly two of these modules are ordinary.

\subsection{A singular vector in \texorpdfstring{$V^{-3}(\g)$}{V(-3,g)}}\label{sec:sing}

Put
$$
 \theta=\epsilon_1+\epsilon_2,\qquad
 \eta=\epsilon_1+\epsilon_3,
$$
and
\begin{align*}
\gamma_1&=\tfrac12(\epsilon_1-\epsilon_2+\epsilon_3-\epsilon_4),&
\gamma_2&=\tfrac12(\epsilon_1+\epsilon_2+\epsilon_3+\epsilon_4),\\
\gamma_3&=\tfrac12(\epsilon_1-\epsilon_2+\epsilon_3+\epsilon_4),&
\gamma_4&=\tfrac12(\epsilon_1+\epsilon_2+\epsilon_3-\epsilon_4),\\
\gamma_5&=\tfrac12(\epsilon_1+\epsilon_2-\epsilon_3-\epsilon_4),&
\gamma_6&=\tfrac12(\epsilon_1+\epsilon_2-\epsilon_3+\epsilon_4).
\end{align*}
Define $v''\in S^3(\g)$ by
\begin{align}
v''={}&
 -2f_{\epsilon_2-\epsilon_3}e_\theta^2
 -2e_{\epsilon_3-\epsilon_4}e_{\epsilon_1+\epsilon_4}e_\theta
 -e_{\epsilon_3}e_{\epsilon_1}e_\theta +2e_{\epsilon_3+\epsilon_4}e_{\epsilon_1-\epsilon_4}e_\theta\notag\\
&
 +2e_{\epsilon_2-\epsilon_3}e_{\eta}^2 +2e_{\epsilon_2-\epsilon_4}e_{\epsilon_1+\epsilon_4}e_{\eta}
 +e_{\epsilon_2}e_{\epsilon_1}e_{\eta}
 -2e_{\epsilon_2+\epsilon_4}e_{\epsilon_1-\epsilon_4}e_{\eta}
 \notag\\
&+2e_{\epsilon_2+\epsilon_3}e_{\epsilon_1-\epsilon_2}e_\theta
 -2e_{\epsilon_2+\epsilon_3}e_{\epsilon_1-\epsilon_3}e_{\eta}
 +4e_{\epsilon_2+\epsilon_3}e_{\epsilon_1-\epsilon_4}
       e_{\epsilon_1+\epsilon_4}\notag\\
&
 -e_{\epsilon_2+\epsilon_3}e_{\epsilon_1}^2 -e_{\gamma_1}e_{\gamma_2}e_\theta
 +e_{\gamma_3}e_{\gamma_4}e_\theta
 +e_{\gamma_5}e_{\gamma_2}e_{\eta}
 -e_{\gamma_6}e_{\gamma_4}e_{\eta} \notag\\
&+e_{\gamma_4}^2e_{\epsilon_1+\epsilon_4}
 -e_{\gamma_4}e_{\gamma_2}e_{\epsilon_1}
 +e_{\gamma_2}^2e_{\epsilon_1-\epsilon_4}
 +2e_{\eta}e_\theta h_{\epsilon_2-\epsilon_3} .
\label{eq:3:singular-symbol}
\end{align}
Define $\widehat v''\in V^{-3}(\g)$ as follows. In each monomial
of \eqref{eq:3:singular-symbol}, replace every factor $a\in\g$ by $a(-1)$, keep
the displayed order, and apply the resulting product to $\vac$.

\begin{theorem}\label{thm:singular}
The vector
\begin{equation}
 v=10\bigl(e_\eta(-2)e_\theta(-1)
           -e_\theta(-2)e_\eta(-1)\bigr)\vac+\widehat v''
\label{eq:3:singular-vector}
\end{equation}
is a $\widehat{\g}$-singular vector in $V^{-3}(\g)$.
Its conformal weight is $3$ and its $\g$-weight is
$$
 \omega_2=2\epsilon_1+\epsilon_2+\epsilon_3,
$$
and its image in $R_{V^{-3}(\g)}$ is $v''$.
\end{theorem}

\begin{proof}
A direct calculation gives
$$
 e_{\alpha_i}(0)v=0\quad(1\leq i\leq4),\qquad
 f_\theta(1)v=0.
$$
These operators generate $\widehat{\mathfrak n}_+$, so $v$ is
singular. Its conformal weight and $\g$-weight follow from \eqref{eq:3:singular-vector}.
Its image in $R_{V^{-3}(\g)}$ is $v''$, since the terms
containing a mode $-2$ vanish.
\end{proof}

Let $J$ be the ideal of $V^{-3}(\g)$ generated by $v$, and put
$$
 \widetilde L_{-3}(\g)=V^{-3}(\g)/J.
$$

\subsection{Irreducible modules in category \texorpdfstring{$\mathcal O$}{O}}

Under the identification $A(V^{-3}(\g))\cong U(\g)$,
\begin{equation}
 [x_1(-n_1)\cdots x_s(-n_s)\vac]
 =(-1)^{\sum_i(n_i-1)}x_s\cdots x_1.
\label{eq:3:zhu-image}
\end{equation}
Let $v'\in U(\g)$ be the image of $v$, and put
$$
 R=U(\g)_{\operatorname{ad}}v',\qquad
 R_0=R^{\mathfrak h}.
$$
For $r\in R_0$, set $p_r=\HC(r)\in S(\mathfrak h)$. Then
$$
 rv_\lambda=p_r(\lambda)v_\lambda,
$$
where $v_\lambda$ is a highest-weight vector of
$V_\g(\lambda)$. The following lemma follows from Zhu's correspondence
and \cite{Ad,AM,ArakawaAdmissible}.

\begin{lemma}\label{lem:zhu-criterion}
An irreducible highest-weight $V^{-3}(\g)$-module
$L_{-3}(\lambda)$ is a $\widetilde L_{-3}(\g)$-module if and
only if
$$
 p_r(\lambda)=0\qquad(r\in R_0).
$$
\end{lemma}

\begin{proposition}\label{prop:zhu-computation}
One has
\begin{equation}
 \dim R_0=26,\qquad
 \dim\{p_r:r\in R_0\}=20.
\label{eq:3:harish-chandra-dimensions}
\end{equation}
A basis $p_1,\ldots,p_{20}$ is listed in
Appendix~\ref{app:polynomials}.  Its common zero set consists of
\begin{align*}
&0,\quad \omega_4,\quad -2\omega_1,\quad
 -\tfrac32\omega_1,\quad
 -\tfrac52\omega_1+\omega_3,\quad
 -3\omega_1+\omega_4,\\
&-\tfrac12\omega_1-\tfrac12\omega_2,\quad
 -\omega_2,\quad -\tfrac12\omega_2,\quad
 -\tfrac32\omega_2+\omega_3,\\
&-2\omega_3+\omega_4,\quad -3\omega_4,\quad
 \tfrac12\omega_1-\tfrac32\omega_2+\omega_3,\quad
 \omega_1-2\omega_2+\omega_4.
\end{align*}
\end{proposition}

\begin{proof}
A direct calculation in Mathematica gives \eqref{eq:3:harish-chandra-dimensions} and the basis
listed in Appendix~\ref{app:polynomials}. The common zero set
is obtained from the Gr\"obner basis in Appendix~\ref{app:groebner}.
\end{proof}

\begin{theorem}\label{thm:maximal}
The ideal $J$ is the maximal proper ideal of $V^{-3}(\g)$.
Consequently,
$$
 \widetilde L_{-3}(\g)=L_{-3}(\g),
$$
and the fourteen weights of Proposition~\ref{prop:zhu-computation}
classify the irreducible $L_{-3}(\g)$-modules in category
$\mathcal O$.
\end{theorem}

\begin{proof}
Suppose that $\widetilde L_{-3}(\g)$ has a nonzero proper graded
ideal $K$, and let $N>0$ be minimal with $K_N\neq0$.
Then $x(n)K_N=0$ for $x\in\g$ and $n>0$, and $K_N$
is an $A(\widetilde L_{-3}(\g))$-module. Let $u\in K_N$ be
a highest-weight vector of an irreducible $\g$-summand, with
highest weight $\mu$. Then $\mu$ is dominant integral and
$p_r(\mu)=0$ for every $r\in R_0$. Proposition~\ref{prop:zhu-computation}
gives $\mu=0$ or $\mu=\omega_4$. The Sugawara formula shows
that $L_0$ acts on $u$ as the quadratic Casimir divided by
$2(-3+h^\vee)$. Hence
\begin{equation}
 N=\frac{(\mu\,|\,\mu+2\rho)}{2(-3+h^\vee)},\qquad h^\vee=9.
\label{eq:3:ideal-lowest-weight}
\end{equation}
For $\mu=0$, \eqref{eq:3:ideal-lowest-weight} gives $N=0$.  For
$\mu=\omega_4=\epsilon_1$, it gives $N=1$, whereas
$\widetilde L_{-3}(\g)_1=\g\cong V_\g(\omega_1)$ contains no
copy of $V_\g(\omega_4)$.  Both alternatives are impossible.
\end{proof}

\begin{corollary}\label{cor:ordinary}
The set 
$
 \{ L_{-3}(\g), L_{-3}(\omega_4) \}$ provides a complete list of irreducible ordinary   $L_{-3}(\g)$-modules.
\end{corollary}

\section{The associated variety of $L_{-3}(\fF)$}\label{sec:ass}

We first prove that $L_{-3}(\g)$ is quasi-lisse.
We then use Drinfeld--Sokolov reduction and the conformal embedding
$L_{-3}(\g)\hookrightarrow L_{-3}(\eE)$ to show that the associated variety 
$X_{L_{-3}(\g)}$ is $\overline{\OO}_{\widetilde A_1}$.

\subsection{\texorpdfstring{$L_{-3}(\g)$}{L(-3,g)} is quasi-lisse}
Let
$$
 W=U(\g)v''\cong V_\g(\omega_2),\qquad
 I_W=S(\g)W.
$$
Here $\g$ acts on $S(\g)$ by the adjoint action. The image of $J$
in $R_{V^{-3}(\g)}=S(\g)$ is $I_W$. Theorem~\ref{thm:maximal}
therefore gives
\begin{equation}
 R_{L_{-3}(\g)}\cong S(\g)/I_W.
\label{eq:4:c2-presentation}
\end{equation}

\begin{proposition}\label{prop:cartan-certificate}
Let $\Psi:S(\g)\to S(\mathfrak h)$ be the algebra homomorphism
defined by $\Psi(h)=h$ for $h\in\mathfrak h$ and
$\Psi(e_\alpha)=\Psi(f_\alpha)=0$ for every positive root $\alpha$.
Then
\begin{equation}
 \Psi(W_0)=S^3(\mathfrak h),\qquad W_0=W^{\mathfrak h}.
\label{eq:4:cartan-image}
\end{equation}
\end{proposition}

\begin{proof}
For $d\geq0$, let $U_{\leq d}(\g)$ be the subspace of $U(\g)$
spanned by products of at most $d$ elements of $\g$. Identify
$U_{\leq3}(\g)/U_{\leq2}(\g)$ with $S^3(\g)$ by sending the
class of $abc$ to the commutative product $abc$ for $a,b,c\in\g$.
Denote the resulting map by
$$
 \sigma_3:U_{\leq3}(\g)\longrightarrow S^3(\g).
$$
It commutes with the adjoint action of $\g$. Equations \eqref{eq:3:singular-vector} and
\eqref{eq:3:zhu-image} give
$$
 \sigma_3(v')=v'',\qquad \sigma_3(R_0)=W_0.
$$

Choose $r_i\in R_0$ with $\HC(r_i)=p_i$, where
$p_1,\ldots,p_{20}$ are the polynomials in
Appendix~\ref{app:polynomials}. Put $w_i=\sigma_3(r_i)$, and let
$p_i^{(3)}$ be the homogeneous component of degree three of $p_i$.
Write $r_i$ as a linear combination of products in which
the factors from $\mathfrak n_-$ precede those from
$\mathfrak h$, and the factors from $\mathfrak h$ precede
those from $\mathfrak n_+$.
The Harish--Chandra projection removes the products
containing a factor from $\mathfrak n_-$ or $\mathfrak n_+$.
After applying $\sigma_3$, the map $\Psi$ removes
the same products. Hence
\begin{equation}
 p_i^{(3)}=\Psi(w_i).
\label{eq:4:cartan-leading-symbols}
\end{equation}
With respect to the lexicographically ordered monomial basis of
$S^3(\mathfrak h)$, with $h_1>h_2>h_3>h_4$, the coefficient
matrix of $p_1^{(3)},\ldots,p_{20}^{(3)}$ has determinant
$-2^{41}3^{10}5^7\neq0$. Therefore, these polynomials form a basis of $S^3(\mathfrak h)$. Since $w_i\in W_0\subset S^3(\g)$, equation \eqref{eq:4:cartan-leading-symbols} proves the assertion.
\end{proof}

\begin{theorem}\label{thm:quasilisse} We have:
$$
 X_{L_{-3}(\g)}\subset\mathcal N_\g.
$$
In particular, $L_{-3}(\g)$ is quasi-lisse.
\end{theorem}

\begin{proof}
Put $X=X_{L_{-3}(\g)}$, and let $y=y_s+y_n\in X$ be the
Jordan decomposition. Since $y_s\in\overline{G\cdot y}$ and $X$
is closed and $G$-stable, we have $y_s\in X$. After conjugation,
we may assume that $y_s\in\mathfrak h$.

Under the identification $\g\cong\g^*$ given by the invariant
form, $\Psi$ is restriction to $\mathfrak h$. Equations
\eqref{eq:4:c2-presentation}--\eqref{eq:4:cartan-image} imply that every element of $S^3(\mathfrak h)$
vanishes at $y_s$. Thus $y_s=0$, and $X\subset\mathcal N_\g$.
The variety $X$ is a union of finitely many nilpotent orbits,
which are its symplectic leaves. Hence $L_{-3}(\g)$ is quasi-lisse.
\end{proof}

\subsection{Drinfeld--Sokolov reduction and the associated variety}

Let $f=f_{\epsilon_1}\in\OO_{\widetilde A_1}$, and put
$$
 e=e_{\epsilon_1},\qquad x=\tfrac12h_{\epsilon_1}.
$$
Then
$$
 \g^\natural=\mathfrak z_\g(e,2x,f)\cong\mathfrak{sl}_4.
$$
As a $\g^\natural\times\mathfrak{sl}_2$-module,
$$
 \g\cong(\mathbf{15},\mathbf1)\oplus(\mathbf6,\mathbf3)
 \oplus(\mathbf1,\mathbf3)\oplus(\mathbf4,\mathbf2)
 \oplus(\mathbf4^*,\mathbf2).
$$
Here $\mathbf4$ is the natural $\mathfrak{sl}_4$-module, $\mathbf{4}^*$ is its dual, 
$\mathbf6=\bigwedge^2\mathbf4$, and $\mathbf{15}$ is the adjoint
module. The second entry denotes the irreducible
$\mathfrak{sl}_2$-module of the indicated dimension.

The strong generators of $\mathcal W^{-3}(\g,f)$ can be chosen
to span the following $\g^\natural$-modules:
\begin{equation}
\begin{array}{c|c}
\text{conformal weight}&\g^\natural\text{-module}\\
\hline
1&\mathbf{15}\\
3/2&\mathbf4\oplus\mathbf4^*\\
2&\mathbf6\oplus\mathbf1.
\end{array}
\label{eq:4:w-strong-generators}
\end{equation}
See \cite[Table 5, orbit $2$]{WExtensions}.
The generator spanning the trivial module may be chosen to be
the Virasoro vector.
For $a\in\g^\natural$, let
$J^{\{a\}}\in\mathcal W^{-3}(\g,f)_1$
denote the generator corresponding to $a$.

\begin{proposition}\label{prop:brst-ideal}
Let $\mathcal J_v$ be the image of $H^0_{DS,f}(J)$ in
$\mathcal W^{-3}(\g,f)$. Then $\mathcal J_v$ contains
$J^{\{a\}}$ for every $a\in\g^\natural$ and the strong generators
spanning $\mathbf4\oplus\mathbf4^*\oplus\mathbf6$ in \eqref{eq:4:w-strong-generators}.
Consequently,
\begin{equation}
 H^0_{DS,f}(L_{-3}(\g))
 \quad\text{is a quotient of}\quad V^{\mathrm{Vir}}(0,0).
\label{eq:4:virasoro-quotient}
\end{equation}
\end{proposition}

\begin{proof}
Write $\g_j=\{a\in\g:[x,a]=ja\}$ and
$\g_{>0}=\bigoplus_{j>0}\g_j$. For a $V^{-3}(\g)$-module $M$,
let
$$
 C_f^\bullet(M)=M\otimes F^{\mathrm{ne}}\otimes F^{\mathrm{ch}},
 \qquad d_M=Q_{(0)},
$$
be the BRST complex and differential of
\cite[Section~4.1]{ArakawaDS} for $f$ and the grading defined by $x$.
Here $F^{\mathrm{ne}}$ is the neutral free-field vertex algebra
associated with $\g_{1/2}$, and $F^{\mathrm{ch}}$ is the charged
free-field vertex superalgebra associated with
$\g_{>0}\oplus\g_{>0}^*$. The operator $d_M$ maps
$C_f^j(M)$ to $C_f^{j+1}(M)$. Write $\vac_{\mathrm{gh}}$ for
the tensor product of the two vacuum vectors.

Choose a homogeneous basis $(u_i)$ of $\g_{>0}$, and let
$\psi_i^*$ be the charged generators corresponding to its dual
basis. The formula for $Q_{(0)}$ gives
$$
 d_M(m\otimes\vac_{\mathrm{gh}})
 =\sum_i\sum_{n\geq0}
 u_i(n)m\otimes(\psi_i^*)_{(-n-1)}\vac_{\mathrm{gh}}.
$$
The terms of $Q_{(0)}$ acting only on the free-field factors
annihilate $\vac_{\mathrm{gh}}$. In the present grading,
$\g_{>0}\subset\mathfrak n_+$. Since $v$ is singular,
$u_i(n)v=0$ for all $i$ and $n\geq0$. Thus
$v\otimes\vac_{\mathrm{gh}}\in C_f^0(J)$ satisfies
$d_J(v\otimes\vac_{\mathrm{gh}})=0$.

Let $\iota:C_f^\bullet(J)\to C_f^\bullet(V^{-3}(\g))$
be the morphism induced by $J\hookrightarrow V^{-3}(\g)$, and put
$$
 z=H^0(\iota)\bigl([v\otimes\vac_{\mathrm{gh}}]_J\bigr)
 \in\mathcal J_v\subset\mathcal W^{-3}(\g,f),
$$
where $[v\otimes\vac_{\mathrm{gh}}]_J$ denotes the class in
$H^0_{DS,f}(J)=\ker(d_J|_{C_f^0(J)})/d_J(C_f^{-1}(J))$.
The conformal weight of $z$ is $3-\omega_2(x)=1$, and it is a
$\g^\natural$-highest-weight vector of weight
$\beta=\epsilon_2+\epsilon_3$, provided that $z\neq0$.

By \cite[Theorem 4.5.9(i) and Corollary 4.5.10(i)]{ArakawaDS},
$$
 R_{\mathcal W^{-3}(\g,f)}\cong\C[\mathcal S_f],
 \qquad \mathcal S_f=f+\g^e.
$$
Under this identification, the image of $z$ is
$v''|_{\mathcal S_f}$. Since $f_\beta\in\g^\natural\subset\g^e$,
the line $f+t f_\beta$ lies in $\mathcal S_f$. Only the term
$-e_\beta e_{\epsilon_1}^2$ in \eqref{eq:3:singular-symbol} contributes on this line.
Thus
$$
 v''(f+t f_\beta)
 =-(e_{\epsilon_1} |\, f)^2(e_\beta\, |\, t f_\beta)
 =-4t.
$$
It follows that $z\neq0$.

By \eqref{eq:4:w-strong-generators}, the space $\mathcal W^{-3}(\g,f)_1$ is the adjoint
$\g^\natural$-module. Its highest-weight space has weight $\beta$
and is one-dimensional. Therefore
$z=cJ^{\{e_\beta\}}$ for some $c\neq0$. Since $\mathcal J_v$
is stable under $\g^\natural$, we obtain
$$
 J^{\{a\}}\in\mathcal J_v\qquad(a\in\g^\natural).
$$
For the strong generators $W^u$ spanning one of the modules
$U=\mathbf4,\mathbf4^*,\mathbf6$, we have
$$
 J^{\{a\}}_{(0)}W^u=W^{[a,u]}.
$$
Since $[\g^\natural,U]=U$, all these generators belong to
$\mathcal J_v$. The quotient is therefore generated by the
image of the Virasoro vector. Its central charge is zero, since
$$
 c(k,\widetilde A_1)=-\frac{6(k+3)(2k+5)}{k+9}.
$$
Exactness of $H^0_{DS,f}$ on $KL_{-3}$ gives
$$
 H^0_{DS,f}(L_{-3}(\g))
 \cong\mathcal W^{-3}(\g,f)/\mathcal J_v
$$
by \cite[Theorem 4.5.7]{ArakawaDS}. This proves \eqref{eq:4:virasoro-quotient}.
\end{proof}

\begin{theorem}\label{thm:collapse} We have:
$$
 \mathcal W_{-3}(\g,f)\cong\C.
$$
\end{theorem}

\begin{proof}
The vertex subalgebra of $\mathcal W^{-3}(\g,f)$ generated
by $J^{\{a\}}$, $a\in\g^\natural$, is a quotient of
$V^0(\g^\natural)$, where $\g^\natural\cong\mathfrak{sl}_4$.
In particular,
\[
 J^{\{a\}}_{(1)}J^{\{b\}}=0
 \qquad(a,b\in\g^\natural).
\]
Since $\mathcal W^{-3}(\g,f)_1$ is spanned by the primary
currents $J^{\{a\}}$, we have
$L_1\mathcal W^{-3}(\g,f)_1=0$.
Together with $\mathcal W^{-3}(\g,f)_0=\C\mathbf1$,
this implies that $\mathcal W^{-3}(\g,f)$ admits a unique
invariant bilinear form $\langle\cdot,\cdot\rangle$
normalized by $\langle\mathbf1,\mathbf1\rangle=1$ \cite{Li94}.
By invariance,
\[
 \langle J^{\{a\}},J^{\{b\}}\rangle
 =-\langle\mathbf1,J^{\{a\}}_{(1)}J^{\{b\}}\rangle=0.
\]
Since distinct conformal-weight subspaces are orthogonal,
$\mathcal W^{-3}(\g,f)_1$ lies in the radical of this form,
which is the maximal proper graded ideal.

The relations $J^{\{a\}}_{(0)}W^u=W^{[a,u]}$ then imply that
the generators spanning $\mathbf4\oplus\mathbf4^*\oplus\mathbf6$
also lie in this ideal. Thus $\mathcal W_{-3}(\g,f)$ is a simple
quotient of $V^{\mathrm{Vir}}(0,0)$ and is isomorphic to $\C$.
\end{proof}

\begin{theorem}\label{thm:associated} We have:
\begin{equation}
 X_{L_{-3}(\g)}
 =\overline{\OO}_{\widetilde A_1}
 =\OO_{\widetilde A_1}\sqcup\OO_{A_1}\sqcup\{0\}.
\label{eq:4:associated-variety}
\end{equation}
\end{theorem}

\begin{proof}
Recall the Dynkin diagram involution $\sigma$ of $\eE$, with
$\eE^\sigma=\g$. The conformal embedding of \cite{AdamovicPerse}
gives the decomposition
$$
 L_{-3}(\eE)
 =L_{-3}(\g)\oplus L_{-3}(\omega_4)
$$
as an $L_{-3}(\g)$-module. The embedding induces a morphism
$$
 \pi:X_{L_{-3}(\eE)}\longrightarrow X_{L_{-3}(\g)},
$$
given by the projection $a\mapsto(a+\sigma(a))/2$ under the
identifications with subvarieties of $\eE$ and $\g$.
By \cite{ArakawaMoreau},
$$
 X_{L_{-3}(\eE)}=\overline{\OO}_{\min}^{E_6}.
$$
Choose a root vector $a\in\eE$ whose root is not fixed by $\sigma$.
Then $\pi(a)$ is a nonzero short-root vector of $\g$. Since
$a\in\OO_{\min}^{E_6}$, we obtain
\begin{equation}
 \overline{\OO}_{\widetilde A_1}
 \subset X_{L_{-3}(\g)}.
\label{eq:4:orbit-inclusion}
\end{equation}

Proposition~\ref{prop:brst-ideal} implies
$$
 \dim X_{H^0_{DS,f}(L_{-3}(\g))}\leq1.
$$
The inclusion \eqref{eq:4:orbit-inclusion} shows that the reduction is nonzero.
By \cite[Theorem 4.21]{ArakawaDS},
\begin{equation}
 X_{H^0_{DS,f}(L_{-3}(\g))}
 =X_{L_{-3}(\g)}\cap\mathcal S_f.
\label{eq:4:reduction-slice}
\end{equation}
By Theorem~\ref{thm:quasilisse}, $X_{L_{-3}(\g)}$ is a union
of nilpotent orbit closures. Suppose that it contains an orbit
$\OO$ not contained in $\overline{\OO}_{\widetilde A_1}$.
The closure order in type $\fF$ gives
$\OO_{\widetilde A_1}\subset\overline{\OO}$
\cite[Chapter 8]{CollingwoodMcGovern}. Since $\mathcal S_f$
is transverse to $\OO_{\widetilde A_1}$ at $f$,
$$
 \dim(\overline{\OO}\cap\mathcal S_f)
 =\dim\OO-\dim\OO_{\widetilde A_1}\geq2,
$$
contradicting \eqref{eq:4:reduction-slice} and the preceding dimension bound.
Together with \eqref{eq:4:orbit-inclusion}, this proves \eqref{eq:4:associated-variety}. The orbit dimensions are
$\dim\OO_{\widetilde A_1}=22$ and $\dim\OO_{A_1}=16$.
\end{proof}

\section{The modular linear differential equation for $L_{-3}(\g)$ }\label{sec:MLDE}

In this section we compute the Poisson coinvariants of $R_{L_{-3}(\g)}$ and derive a fifth-order modular linear differential equation satisfied by the
characters of all ordinary $L_{-3}(\g)$-modules.
We prove that this order is minimal for the vacuum character among
monic holomorphic $\mathrm{SL}_2(\mathbb Z)$-MLDEs.

\subsection{Poisson coinvariants}\label{sec:poisson}

Put $R=R_{L_{-3}(\g)}$.  Since $R$ is generated by the image of
$\g$,
$$
 \{R,R\}=\{\g,R\}=\g R.
$$
Every homogeneous component is a finite-dimensional semisimple
$\g$-module.  Therefore
$$
 R/\{R,R\}\cong R^\g
$$
as graded vector spaces.

Choose invariant generators
\begin{equation}
 S(\g)^\g=\C[Q,P_6,P_8,P_{12}],\qquad
 \deg(Q,P_6,P_8,P_{12})=(2,6,8,12).
\label{eq:5:invariant-generators}
\end{equation}
For $d\geq3$, define
$$
 \delta_d:S^d(\g)^\g\longrightarrow S^{d-3}(\g),\qquad
 \delta_d(F)=v''(\partial)F,
$$
using the invariant form to identify $\g\cong\g^*$.

\begin{lemma}\label{lem:apolar}
\begin{equation}
 \dim(R^\g)_d=\dim\ker\delta_d.
\label{eq:5:apolar-kernel}
\end{equation}
\end{lemma}

\begin{proof}
Under the invariant apolar pairing on $S^d(\g)$, the orthogonal
complement of $WS^{d-3}(\g)$ consists of the elements annihilated
by $w(\partial)$ for all $w\in W$.  For an invariant $F$, the
map $w\mapsto w(\partial)F$ is $\g$-equivariant.  Since $v''$
generates $W$, it vanishes exactly when $\delta_d(F)=0$.
Exactness of invariants in each finite-dimensional homogeneous
component proves \eqref{eq:5:apolar-kernel}.
\end{proof}

\begin{proposition}\label{prop:ranks}
With the invariant bases
$$
\begin{array}{c|l}
d&\text{basis of }S^d(\g)^\g\\
\hline
6&Q^3,\ P_6\\
8&Q^4,\ QP_6,\ P_8\\
10&Q^5,\ Q^2P_6,\ QP_8\\
12&Q^6,\ Q^3P_6,\ Q^2P_8,\ P_6^2,\ P_{12},
\end{array}
$$
the ranks of $\delta_d$ are
\begin{equation}
 \operatorname{rank}\delta_6=1,\quad
 \operatorname{rank}\delta_8=2,\quad
 \operatorname{rank}\delta_{10}=3,\quad
 \operatorname{rank}\delta_{12}=5.
\label{eq:5:apolar-ranks}
\end{equation}
Moreover,
\begin{equation}
 Q^3,Q^4\notin I_W,\qquad Q^5\in I_W.
\label{eq:5:casimir-powers}
\end{equation}
\end{proposition}

\begin{proof}
The exact certificates for \eqref{eq:5:apolar-ranks}--\eqref{eq:5:casimir-powers} are given in
Appendix~\ref{app:coinvariant-certificate}.  The upper bounds in
degrees $6$ and $8$ are
$$
 [S^3(\g):V_\g(\omega_2)]=1,\qquad
 [S^5(\g):V_\g(\omega_2)]=2.
$$
The nonzero minors in the appendix give the matching lower bounds.
In degrees $10$ and $12$, the displayed full-size minors prove
injectivity.  The degree-six invariant obtained by contracting two
copies of $W$ spans $(I_W^\g)_6$ and is not proportional to
$Q^3$.  The degree-eight apolar certificate pairs $Q^4$
nontrivially with $\ker\delta_8$.  Thus $Q^3,Q^4\notin I_W$.
Injectivity of $\delta_{10}$ gives $Q^5\in I_W$.
\end{proof}

\begin{theorem}\label{thm:coinvariants}
There are isomorphisms
\begin{equation}
 R^\g\cong\C[Q]/(Q^5),\qquad
 R/\{R,R\}\cong\C\{1,Q,Q^2,Q^3,Q^4\}.
\label{eq:5:invariant-quotient}
\end{equation}
For the conformal class $\ell=[\omega]$,
\begin{equation}
 R/\{R,R\}\cong\C[\ell]/(\ell^5).
\label{eq:5:poisson-coinvariants}
\end{equation}
\end{theorem}

\begin{proof}
The degree-six and degree-eight rows of \eqref{eq:5:apolar-ranks}, together with \eqref{eq:5:casimir-powers},
show that $P_6$ and $P_8$ are scalar multiples of $Q^3$ and
$Q^4$ in $R^\g$.  The degree-ten row gives $Q^5=0$, and the
degree-twelve row gives $P_{12}=0$.  Equation \eqref{eq:5:invariant-generators} now proves
\eqref{eq:5:invariant-quotient}.  The Sugawara formula identifies $\ell$ with a nonzero
scalar multiple of $Q$, giving \eqref{eq:5:poisson-coinvariants}.
\end{proof}

\subsection{MLDE for $L_{-3}(\fF)$}\label{sec:mlde}

Let
$$
 A=\C[\widetilde G_4,\widetilde G_6],\qquad V=L_{-3}(\g),
$$
and define Zhu's genus-one quotient
$$
 \cH(V)=
 \frac{V\otimes A}{[V\otimes A,V\otimes A]+O_q(V)}.
$$
By \cite[Proposition 5.2]{ArakawaKawasetsu}, their filtration gives a
surjection
$$
 \left(R/\{R,R\}\right)\otimes A
 \twoheadrightarrow\operatorname{gr}\cH(V)
$$

\begin{proposition}\label{prop:genus-one}
The $A$-module $\cH(V)$ is generated by
\begin{equation}
 \vac,\quad L[-2]\vac,\quad L[-2]^2\vac,\quad
 L[-2]^3\vac,\quad L[-2]^4\vac.
\label{eq:5:genus-one-generators}
\end{equation}
\end{proposition}

\begin{proof}
Their principal symbols are nonzero scalar multiples of
$1,\ell,\ell^2,\ell^3,\ell^4$.  Theorem
\ref{thm:coinvariants} and induction on the filtration degree give
\eqref{eq:5:genus-one-generators}.
\end{proof}

For a modular form $f$ with weight $k\in2\mathbb Z_{\geq0}$, define the Serre derivative
\begin{equation}
 \D f=\left(q\frac{d}{dq}-\frac{k}{12}E_2\right)f,\qquad
 \D^{(j+1)}f=\D\bigl(\D^{(j)}f\bigr),\quad(j\geq0).
\label{eq:5:serre-derivatives}
\end{equation}
Each application of $\D$ raises the weight by two. We suppress
the weight subscript throughout and assign input weight zero to characters
in the MLDEs. In a composition, the left factor acts at the output weight
of the right factor, thus $\D^{(j)}$ denotes the iterated covariant
derivative with these successive weight shifts.

\begin{theorem}\label{thm:mlde-existence}
The character of every ordinary $V$-module satisfies an equation
\begin{equation}
 \left(
 \D^{(5)}+aE_4\D^{(3)}+bE_6\D^{(2)}
 +cE_4^2\D+dE_4E_6
 \right)f=0
\label{eq:5:mlde-ansatz}
\end{equation}
for constants $a,b,c,d$.
\end{theorem}

\begin{proof}
By Proposition~\ref{prop:genus-one}, the class of
$L[-2]^5\vac$ is an $A$-linear combination of \eqref{eq:5:genus-one-generators}.  Taking
the homogeneous component of square-bracket weight $10$, and using
$$
 M_2=0,\quad M_4=\C E_4,\quad M_6=\C E_6,\quad
 M_8=\C E_4^2,\quad M_{10}=\C E_4E_6,
$$
gives a monic relation of the form \eqref{eq:5:mlde-ansatz}.  Zhu recursion converts it
to the asserted equation.  Convergence for ordinary characters
follows from quasi-lisseness
\cite[Theorem 5.7]{ArakawaKawasetsu}.
\end{proof}

The Sugawara central charge is
$$
 c=\frac{-3\dim\g}{-3+h^\vee}=-26.
$$
The maximal ideal of $V^{-3}(\g)$ has no nonzero vectors
of conformal weight less than $3$. Hence
\begin{equation}
 \chi_V=q^{13/12}
 \left(1+52q+1430q^2+26286q^3+O(q^4)\right).
\label{eq:5:vacuum-initial-coefficients}
\end{equation}
The $q^3$-coefficient 26286 comes from the difference between the dimension of  conformal-weight three subspace of $V^{-3}(\g)$
$$
52+52^2+\binom{54}{3}=27560,
$$
and $\dim V_\g(\omega_2)=1274$. Substituting \eqref{eq:5:vacuum-initial-coefficients}
into \eqref{eq:5:mlde-ansatz} and comparing the coefficients
of $q^{13/12+n}$ for $n=0,1,2,3$ uniquely determines
$a,b,c,d$. We obtain the explicit MLDE in the following theorem. It is easy to check no smaller order monic MLDE can be satisfied by \eqref{eq:5:vacuum-initial-coefficients}.

\begin{theorem}\label{thm:mlde}
The differential operator in
Theorem~\ref{thm:mlde-existence} is 
\begin{equation}
\begin{aligned}
\mathcal L_5=
\D^{(5)}-\frac{1675}{576}E_4\D^{(3)}
+\frac{965}{2304}E_6\D^{(2)}
+\frac{116545}{82944}E_4^2\D
+\frac{46475}{995328}E_4E_6.
\end{aligned}
\label{eq:5:f4-mlde}
\end{equation}
The equation $\mathcal L_5f=0$ is the unique monic
holomorphic $\mathrm{SL}_2(\mathbb Z)$-MLDE of order five
satisfied by the vacuum character.
No such equation of smaller order is satisfied
by the vacuum character.
\end{theorem}

The indicial polynomial of differential operator $\mathcal{L}_5$ is
$$
 P(x)=\frac{(12x-25)(12x-13)(12x+11)(24x+1)(24x+13)}
 {995328},
$$
and we have a factorization
\begin{equation}\label{eq:5:f4-factorization}
\mathcal L_5=
\left(
\D^{(3)}-\frac{1103}{576}E_4\D-\frac{147}{256}E_6
\right)
\circ
\left(
\D^{(2)}-\frac{143}{144}E_4
\right).    
\end{equation}

\begin{theorem}\label{thm:characters}
The characters of the two irreducible ordinary
$L_{-3}(\g)$-modules have the following expansions:
\begin{equation}\label{eq:5:f4-vacuum-character}
    \chi_{L_{-3}(\g)}
={}q^{13/12}\bigl(
1+52q+1430q^2+26286q^3+361426q^4+3983382q^5+O(q^{6})\bigr),
\end{equation}
and
\begin{equation}\label{eq:5:f4-other-character}
   \chi_{L_{-3}(\omega_4)}
=q^{25/12}\bigl(
26+1079q+22984q^2+336999q^3+3831724q^4 +O(q^{5})\bigr).
\end{equation}
\end{theorem}

\begin{proof}
The relevant indicial roots are $13/12$ and $25/12$.  For $n\geq2$,
$P(13/12+n)\neq0.$ 
Thus \eqref{eq:5:f4-mlde} and the initial values $a_0=1,a_1=52$ determine
\eqref{eq:5:f4-vacuum-character}.  The lowest conformal-weight subspace of
$L_{-3}(\omega_4)$ is the $\g$-module $V_\g(\omega_4)$,
which has dimension $26$. Moreover, $P(25/12+n)\neq0$
for $n\geq1$. Hence its leading coefficient determines \eqref{eq:5:f4-other-character}.
\end{proof}

Put $V_{E_6}=L_{-3}(E_6)$ and $V_{F_4}=L_{-3}(\g)$,
and write $F'=q\,dF/dq$.
Arakawa and Kawasetsu \cite[Section~7]{ArakawaKawasetsu}
obtained the following formula for the vacuum character
of $V_{E_6}$:
\begin{equation*}
    \chi_{V_{E_6}}=-\frac{1}{462}\bigg[\frac{E_6(\tau)E_4'(\tau)}{240\eta^{22}(\tau)}-\eta^{2}(\tau)\bigg]. 
\end{equation*}
Put 
\begin{equation*}
    T=\left(\frac{\eta(2\tau)}{\eta(\tau)}\right)^{26}.
\end{equation*}
\begin{corollary}\label{cor:F4}
The characters of the two irreducible ordinary
$V_{F_4}$-modules are
\begin{equation}\label{eq:5:f4-character-formulas}
    \chi_{V_{F_4}}=\frac{\chi_{V_{E_6}}+T}{2},\qquad \chi_{L_{-3}(\omega_4)}=\frac{\chi_{V_{E_6}}-T}{2}.
\end{equation}
\end{corollary}
\begin{proof}
The functions on the right-hand side of
\eqref{eq:5:f4-character-formulas} have the initial
coefficients specified in the proof of
Theorem~\ref{thm:characters}.
It remains to show that they are annihilated by
$\mathcal L_5$. Since $\chi_{V_{E_6}}$ is annihilated by
$\D^{(2)}-\frac{143}{144}E_4$, equation
\eqref{eq:5:f4-factorization} gives
$\mathcal L_5\chi_{V_{E_6}}=0$.
It therefore suffices to prove $\mathcal L_5T=0$.
We have another easily verified factorization
\begin{equation*}
    \mathcal{L}_5= \left(\D^{(2)}-\frac{323}{144}E_4\right)
\circ\left(\D^{(3)}-\frac{383}{576}E_4\D
-\frac{169}{6912}E_6\right).
\end{equation*}
Here the right factor has input weight zero and the left factor
has input weight six. Put $E_2^{(2)}(\tau)=2E_2(2\tau)-E_2(\tau)$, which has weight two.
The modular form identities on congruence subgroup $\Gamma(2)$
\begin{equation*}
   \Big( \log\frac{\eta(2\tau)}{\eta(\tau)}\Big)'=\frac{E_2^{(2)}}{24},\qquad \D E_2^{(2)}=\frac{E_4- 2(E_2^{(2)})^2}{6},\qquad E_6= 4(E_2^{(2)})^3-3E_2^{(2)} E_4
\end{equation*}
imply that $T$ is annihilated by
$\D^{(3)}-\frac{383}{576}E_4\D-\frac{169}{6912}E_6$.
Hence $\mathcal L_5T=0$. 

\end{proof}
In fact, the formulas \eqref{eq:5:f4-character-formulas} can also be directly obtained from the conformal embedding $L_{-3}(\fF) \hookrightarrow L_{-3}(\eE)$ exploiting the $\Gamma$-action. In general, finding an explicit modular form expression for a solution of a MLDE is difficult. On the contrary, proving a modular form satisfies a MLDE is straightforward. It should also be emphasized that our method to prove MLDE for ordinary characters works for general quasi-lisse vertex operator algebras, which does not reply on conformal embeddings.

\section{MLDE for $L_{-2}(B_3)$ and $L_{-2}(G_2)$}\label{sec:MLDE2}
\subsection{The conformal embedding chain and the differential operators}
Put
$$
 V_B=L_{-2}(B_3),\qquad V_G=L_{-2}(G_2),\qquad
 V_D=L_{-2}(D_4).
$$
These vertex algebras all have central charge $-14$ and occur in the conformal embedding chain
\begin{equation*}
 V_G\hookrightarrow V_B\hookrightarrow V_D
\end{equation*}
of Adamovi\'c and Per\v se \cite[Theorems~4.1 and~5.4]{AdamovicPerse}.
The defining ideals and associated varieties of $V_B$ and $V_G$ are studied in
\cite{ADFLM}. In particular, both vertex algebras are quasi-lisse.
Earlier $G_2$ singular-vector and Zhu-algebra calculations at the
admissible levels $-5/3,-4/3,-2/3$ appear in \cite{AxtellLee}.
We write $L^B_{-2}(\lambda)$ and $L^G_{-2}(\lambda)$ when the finite Lie
algebra needs to be specified. In Bourbaki numbering, their irreducible
ordinary modules have the following lowest spaces:
\begin{equation}\label{eq:6:ordinary-module-data}
\begin{array}{c|c|c|c}
\text{algebra}&\text{finite highest weights}&\text{top dimensions}
 &\text{conformal weights}\\ \hline
V_B&0,\ \omega_1&1,\ 7&0,\ 1\\
V_G&0,\ \omega_1,\ \omega_2&1,\ 7,\ 14&0,\ 1,\ 2.
\end{array}
\end{equation}
The classification is given in \cite[Theorems~3.5, 3.6 and~C]{ADFLM};
the conformal weights follow from \eqref{eq:2:conformal-weight}.
For $V_D$, the vacuum is the only irreducible ordinary module and
the ordinary category is completely reducible
\cite[Theorem~4.3]{PerseD}. Complete reducibility for $V_B$ also
follows from \cite[Theorem~5.9(2)]{AKMPP}.

Use the Serre derivatives of \eqref{eq:5:serre-derivatives}, and define the following modular linear differential operators
\begin{align}
 \mathcal L_2&=\D^{(2)}-\frac{35}{144}E_4,
 \\
 \mathcal L_3&=\D^{(3)}-\frac{527}{576}E_4\D
                         -\frac{25}{256}E_6,
\end{align}
and
\begin{equation}
\begin{split}
 \mathcal L_6={}&\D^{(6)}
 +\mu_1E_4\D^{(4)}+\mu_2E_6\D^{(3)}
 +\mu_3E_4^2\D^{(2)}\\
 &+\mu_4E_4E_6\D+\mu_5E_6^2+\mu_6E_4^3,
\end{split}
\end{equation}
where
\begin{equation*}
\begin{aligned}
 \mu_1&=-\frac{13349851}{2770416},&
 \mu_2&=\frac{35698963}{37400616},&
 \mu_3&=\frac{19166721109}{3590459136},\\
 \mu_4&=\frac{1824013379}{5385688704},&
 \mu_5&=\frac{2493803053}{5385688704},&
 \mu_6&=-\frac{141756938771}{172342038528}.
\end{aligned}
\end{equation*}
In the compositions below, $\mathcal L_2$, $\mathcal L_3$ and
$\mathcal L_6$ have input weights $0$, $4$ and $10$, respectively,
as prescribed by the convention in \eqref{eq:5:serre-derivatives}.

Arakawa and Kawasetsu \cite[Section~7]{ArakawaKawasetsu} prove that the
vacuum character of $V_D$ satisfies
\begin{equation}\label{eq:6:d4-mlde}
 \mathcal L_2\chi_{V_D}=0.
\end{equation}
Write $f_B=\chi_{V_B}$, $f_G=\chi_{V_G}$ and $f_D=\chi_{V_D}$, and put
$$
 \mathcal P_B=\mathcal L_3\circ\mathcal L_2,\qquad
 \mathcal P_G=\mathcal L_6\circ\mathcal P_B.
$$
For the vacuum characters $f_B$ of $V_B$ and $f_G$ of $V_G$, we have
\begin{equation}\label{eq:6:vacuum-identities}
 \mathcal P_Bf_B=0,\qquad \mathcal P_Gf_G=0.
\end{equation}
These are proved completely parallel to the proof of $L_{-3}(F_4)$ case. First, we use the Poisson-coinvariant method to find the orders of the monic MLDEs. Then we compute the vacuum character to a sufficient high $q$-degree and use the Fourier coefficients to determine the parameters in the modular linear differential operator.\footnote{An efficient SageMath program was developed in \cite{Duque:2022tub} to compute the ordinary module characters of affine VOA at negative integer levels. We originally used such program to calculate the vacuum characters and obtained the MLDEs for $L_{-3}(F_4)$, $L_{-2}(G_2)$ and $L_{-2}(B_3)$ in 2024. For example, for $L_{-2}(G_2)$ we computed the vacuum character to $q^{15}$-degree to find the 11th order MLDE.}

The orders five and eleven are established in
Appendix~\ref{app:B3G2} using Poisson coinvariants and Zhu recursion.
Here we use equations \eqref{eq:6:vacuum-identities},
the branching rules \eqref{eq:6:branching-characters}, and the
defining relations given by the singular vector of degree six
for $G_2$ and, for $B_3$, by the singular vector of degree two
together with the subsingular vector of degree six
\cite[Theorems~3.6 and~5.8]{ADFLM}.


\subsection{Differential equations for all ordinary modules}
We combine equations \eqref{eq:6:vacuum-identities}
with the branching rules \eqref{eq:6:branching-characters}
to obtain differential equations for the characters
of ordinary $V_B$-modules and $V_G$-modules.

\begin{theorem}\label{thm:mlde-b3}
The character of every ordinary $L_{-2}(B_3)$-module satisfies a monic MLDE of order five:
\begin{equation}\label{eq:6:b3-mlde}
 (\mathcal L_3\circ\mathcal L_2)\chi_M=0.
\end{equation}
\end{theorem}

\begin{theorem}\label{thm:mlde-g2}
The character of every ordinary $L_{-2}(G_2)$-module satisfies a monic  MLDE of order eleven
\begin{equation}\label{eq:6:g2-mlde}
 (\mathcal L_6\circ\mathcal L_3\circ\mathcal L_2)\chi_M=0.
\end{equation}
\end{theorem}

\begin{proof}[Proof of Theorems~\ref{thm:mlde-b3} and~\ref{thm:mlde-g2}]
Let $b_1=\chi_{L^B_{-2}(\omega_1)}$ and
$g_i=\chi_{L^G_{-2}(\omega_i)}$ for $i=1,2$.
The branching formulas
\cite[Theorem~4.1 and Corollary~5.5]{AdamovicPerse} give
\begin{equation}\label{eq:6:branching-characters}
 f_D=f_B+b_1,\qquad f_B=f_G+g_1,\qquad b_1=g_1+g_2.
\end{equation}
Equivalently,
$$
 b_1=f_D-f_B,\qquad g_1=f_B-f_G,\qquad
 g_2=f_D-2f_B+f_G.
$$
By \eqref{eq:6:d4-mlde}, both $\mathcal P_B$ and $\mathcal P_G$
annihilate $f_D$. The right factorization and
\eqref{eq:6:vacuum-identities} also give $\mathcal P_Gf_B=0$.
Thus $\mathcal P_B$ annihilates $f_B,b_1$, and $\mathcal P_G$
annihilates $f_G,g_1,g_2$. These exhaust the irreducible ordinary
characters by \eqref{eq:6:ordinary-module-data}.

To pass to all ordinary modules, choose a highest-weight vector in
the lowest nonzero eigenspace of any $L_0$-coset modulo $\mathbb Z$.
Its cyclic module has a simple ordinary highest-weight quotient
with the same lowest conformal weight. Consequently, every nonzero
ordinary subquotient has a nonzero eigenspace at one
of the possible simple lowest conformal weights: $0,1$ for $B_3$
and $0,1,2$ for $G_2$. The sum of the dimensions of these eigenspaces
is finite and decreases in a proper chain of ordinary subquotients.
Hence every ordinary module has finite length. Additivity of
characters now proves both statements without a semisimplicity
assumption. Convergence on the upper half-plane follows from
quasi-lisseness \cite[Theorem~5.7]{ArakawaKawasetsu}.
Finally, the leading coefficient of every factor is one, which
gives the asserted orders and modular weights.
\end{proof}

Appendix~\ref{app:B3G2} will also present a different proof for Theorem \ref{thm:mlde-b3}. 

Arakawa and Kawasetsu \cite[Section~7]{ArakawaKawasetsu}
obtained the following formula for the vacuum character of $V_D$:
\begin{equation*}
    \chi_{V_D}=\frac{E_4'(\tau)}{240\eta^{10}(\tau)}
=\frac{E_2(\tau)E_4(\tau)-E_6(\tau)}{720\eta^{10}(\tau)}. 
\end{equation*}
Put
\begin{equation*}
    T_1=\left(\frac{\eta(2\tau)}{\eta(\tau)}\right)^{14},\qquad T_2=\left(\frac{\eta(3\tau)}{\eta(\tau)}\right)^{7}.
\end{equation*}
\begin{corollary}\label{cor:B3}
The characters of the two irreducible ordinary
$V_B$-modules are
$$
\chi_{V_B}=\frac{\chi_{V_D}+T_1}{2},\qquad \chi_{L^B_{-2}(\omega_1)}=\frac{\chi_{V_D}-T_1}{2}.
$$
\end{corollary}
\begin{corollary}\label{cor:G2}
The characters of the three irreducible ordinary
$V_G$-modules are
\begin{equation*}
    \chi_{V_G}=\frac{\chi_{V_D}+3T_1+2T_2}{6},\quad\ 
\chi_{L^G_{-2}(\omega_1)}=\frac{\chi_{V_D}-T_2}{3},\quad\ 
\chi_{L^G_{-2}(\omega_2)}=\frac{\chi_{V_D}-3T_1+2T_2}{6}.
\end{equation*}
\end{corollary}
The proofs of Corollaries \ref{cor:B3} and \ref{cor:G2} are the same as the proof of Corollary \ref{cor:F4}.

\appendix

\section{The Harish--Chandra polynomials}\label{app:polynomials}

The image $v'$ in Zhu's algebra gives the following basis of
Harish--Chandra polynomials.
\begingroup
\raggedbottom
\small
\setlength{\jot}{5pt}
\begin{align*}
p_1&={}h_4(h_4-1)(2h_2+2h_3+h_4+3),\\[0pt]
p_2&={}-h_4(h_3+h_4-1)(2h_2+2h_3+h_4+3),\\[0pt]
p_3&={}-h_4(h_2-2h_4+2)(2h_2+2h_3+h_4+3),\\[0pt]
p_4&={}-h_3h_4(h_3+h_4+1),\\[0pt]
p_5&={}\begin{aligned}[t]
  &-2h_2^2h_4-8h_2h_3^2-18h_2h_3h_4+8h_2h_3-5h_2h_4^2+h_2h_4-4h_3^3\\
  &\quad{}-8h_3^2h_4-4h_3^2-4h_3h_4^2-12h_3h_4+8h_3-2h_4^3-4h_4^2+6h_4,
\end{aligned}\\[0pt]
p_6&={}-h_3(2h_2^2-h_2h_3-6h_2h_4+4h_2-h_3^2-2h_3h_4-h_3+h_4^2-5h_4+2),\\[0pt]
p_7&={}h_4(h_4-1)(2h_1+2h_2+2h_3+h_4+3),\\[0pt]
p_8&={}-h_4(2h_1h_3+2h_1h_4-2h_1+2h_2h_3+2h_2h_4-2h_2+h_3^2+2h_3h_4+h_4^2+2h_4-3),\\[0pt]
p_9&={}\begin{aligned}[t]
  &-h_4(2h_1^2+4h_1h_2+2h_1h_3-3h_1h_4+9h_1+2h_2^2+2h_2h_3\\
  &\quad{}-3h_2h_4+7h_2-4h_3h_4+4h_3-2h_4^2-4h_4+6),
\end{aligned}\\[0pt]
p_{10}&={}\begin{aligned}[t]
  &-2h_1^2h_4-4h_1h_2h_4-8h_1h_3^2-18h_1h_3h_4+8h_1h_3-5h_1h_4^2-h_1h_4\\
  &\quad{}-2h_2^2h_4-8h_2h_3^2-18h_2h_3h_4+8h_2h_3-5h_2h_4^2+h_2h_4-4h_3^3\\
  &\quad{}-8h_3^2h_4-4h_3^2-4h_3h_4^2-12h_3h_4+8h_3-2h_4^3-4h_4^2+6h_4,
\end{aligned}\\[0pt]
p_{11}&={}\begin{aligned}[t]
  &-h_3(2h_1^2+4h_1h_2-h_1h_3-6h_1h_4+6h_1+2h_2^2-h_2h_3\\
  &\quad{}-6h_2h_4+4h_2-h_3^2-2h_3h_4-h_3+h_4^2-5h_4+2),
\end{aligned}\\[0pt]
p_{12}&={}\begin{aligned}[t]
  &-h_1(16h_1h_2+4h_1h_3+2h_1h_4+16h_2^2+8h_2h_3+4h_2h_4\\
  &\quad{}+16h_2+6h_3^2+6h_3h_4+4h_3-11h_4^2+17h_4),
\end{aligned}\\[0pt]
p_{13}&={}-h_4(h_4-1)(4h_1+8h_2+6h_3+3h_4+9),\\[0pt]
p_{14}&={}h_4(4h_1h_2-10h_1h_4+10h_1+4h_2^2-20h_2h_4+24h_2-15h_3h_4+15h_3-8h_4^2-16h_4+24),\\[0pt]
p_{15}&={}\begin{aligned}[t]
  &240h_1^2h_2-240h_1h_2^2-480h_1h_2h_3+688h_1h_2h_4-132h_1h_4^2+132h_1h_4-480h_2^3-480h_2^2h_3\\
  &\quad{}+1428h_2^2h_4-720h_2^2-48h_2h_3^2+1284h_2h_3h_4-432h_2h_3+1026h_2h_4^2+1142h_2h_4\\
  &\quad{}-240h_2+36h_3^3+72h_3^2h_4+36h_3^2+342h_3h_4^2-198h_3h_4-72h_3+185h_4^3+370h_4^2-555h_4,
\end{aligned}\\[0pt]
p_{16}&={}\begin{aligned}[t]
  &-h_4(8h_1h_2+8h_1h_3+8h_1h_4-8h_1+8h_2^2+16h_2h_3\\
  &\quad{}+16h_2h_4-8h_2+6h_3^2+12h_3h_4+5h_4^2+10h_4-15),
\end{aligned}\\[0pt]
p_{17}&={}\begin{aligned}[t]
  &12h_1h_2h_3-40h_1h_2h_4+6h_1h_3^2-40h_1h_3h_4+12h_1h_3-40h_1h_4^2+40h_1h_4\\
  &\quad{}+12h_2^2h_3-40h_2^2h_4+12h_2h_3^2-80h_2h_3h_4+36h_2h_3-80h_2h_4^2+40h_2h_4\\
  &\quad{}+3h_3^3-30h_3^2h_4+18h_3^2-60h_3h_4^2+24h_3-25h_4^3-50h_4^2+75h_4,
\end{aligned}\\[0pt]
p_{18}&={}\begin{aligned}[t]
  &-4524h_1h_2h_3-7600h_1h_2h_4-6582h_1h_3^2-6160h_1h_3h_4-204h_1h_3\\
  &\quad{}-1480h_1h_4^2+1480h_1h_4-4524h_2^2h_3-7600h_2^2h_4-13164h_2h_3^2\\
  &\quad{}-12320h_2h_3h_4-4932h_2h_3-2960h_2h_4^2-4640h_2h_4-6531h_3^3-7050h_3^2h_4\\
  &\quad{}-10026h_3^2-2490h_3h_4^2-5130h_3h_4-408h_3-160h_4^3-320h_4^2+480h_4,
\end{aligned}\\[0pt]
p_{19}&={}\begin{aligned}[t]
  &21600h_1h_2^2+13332h_1h_2h_3+11180h_1h_2h_4+10800h_1h_2-8454h_1h_3^2\\
  &\quad{}-14020h_1h_3h_4-3948h_1h_3-5650h_1h_4^2+5650h_1h_4+32400h_2^3+34932h_2^2h_3\\
  &\quad{}+25130h_2^2h_4+48600h_2^2-7188h_2h_3^2-12650h_2h_3h_4+22716h_2h_3\\
  &\quad{}-10625h_2h_4^2+33505h_2h_4+16200h_2-10707h_3^3-22665h_3^2h_4-18882h_3^2\\
  &\quad{}-20085h_3h_4^2-18765h_3h_4-1416h_3-6085h_4^3-12170h_4^2+18255h_4,
\end{aligned}\\[0pt]
p_{20}&={}\begin{aligned}[t]
  &72000h_1^3+216000h_1^2h_2+144000h_1^2h_3+165000h_1^2h_4+252000h_1^2+86400h_1h_2^2\\
  &\quad{}+237324h_1h_2h_3+289760h_1h_2h_4+367200h_1h_2+117222h_1h_3^2+227960h_1h_3h_4\\
  &\quad{}+325164h_1h_3+3500h_1h_4^2+383500h_1h_4+216000h_1-129600h_2^3-50676h_2^2h_3\\
  &\quad{}-29440h_2^2h_4-194400h_2^2+91884h_2h_3^2-2000h_2h_3h_4+94212h_2h_3\\
  &\quad{}-162200h_2h_4^2+143560h_2h_4-64800h_2+45651h_3^3+14370h_3^2h_4+137826h_3^2\\
  &\quad{}-117870h_3h_4^2+127170h_3h_4+112488h_3-60520h_4^3-121040h_4^2+181560h_4.
\end{aligned}
\end{align*}
\endgroup
\section{A Gr\"obner basis}\label{app:groebner}
With lexicographic order $h_1>h_2>h_3>h_4$, the ideal generated by
the polynomials in Appendix~\ref{app:polynomials} has the following
reduced basis, up to multiplication of individual elements by
nonzero rational scalars:
\begingroup
\small
\begin{align*}
g_1={}&h_1(2h_1^2+7h_1+6h_2-h_3-3h_4+6),\\
g_2={}&4h_1^2h_2+2h_1h_2+6h_1h_3+10h_2h_3
-6h_2h_4+4h_3^2-3h_3h_4+11h_3,\\
g_3={}&h_3(2h_1^2-h_1-10h_2-5h_3-10),\\
g_4={}&h_4(h_1^2+3h_1+2h_2),\\
g_5={}&4h_1h_2^2+2h_1h_2-6h_1h_3-10h_2h_3
+6h_2h_4-4h_3^2+3h_3h_4-11h_3,\\
g_6={}&h_3(4h_1h_2+6h_1+10h_2+5h_3+10),\\
g_7={}&h_2h_4(h_1-1),\\
g_8={}&h_1h_3(h_3-1),\\
g_9={}&h_1h_3h_4,\\
g_{10}={}&h_1h_4(h_4-1),\\
g_{11}={}&h_2(2h_2^2+3h_2-h_3-3h_4+1),\\
g_{12}={}&h_3(2h_2^2+3h_2-h_3-3h_4+1),\\
g_{13}={}&h_2h_4(h_2+2),\\
g_{14}={}&h_3(h_2h_3-h_2+h_3+3h_4-1),\\
g_{15}={}&h_2h_3h_4,\\
g_{16}={}&h_2h_4(h_4-1),\\
g_{17}={}&h_3(h_3^2-h_3-6h_4),\\
g_{18}={}&h_3h_4(h_3+2),\\
g_{19}={}&h_3h_4(h_4-1),\\
g_{20}={}&h_4(h_4-1)(h_4+3).
\end{align*}
\endgroup
Exact division gives zero remainder for every $p_i$ modulo
$(g_1,\ldots,g_{20})$ and for every $g_i$ modulo
$(p_1,\ldots,p_{20})$.  The ideal $(g_1,\ldots,g_{20})$ is zero-dimensional.
Its zero set consists of the fourteen points listed in
Proposition~\ref{prop:zhu-computation}.

\section{Calculations of Poisson coinvariants}
\label{app:coinvariant-certificate}
This section collects some necessary data for the computational proof of  Proposition \ref{prop:ranks}. Order the forty-eight roots lexicographically by their doubled
$\epsilon$-coordinates, let $x_1,\ldots,x_{48}$ be the
corresponding root vectors. Complete the basis by
$x_{48+i}=h_i$ for $1\leq i\leq4$, and define
$$
 T_m(x)=\operatorname{tr}(\operatorname{ad}x)^m
 \qquad(m=2,6,8,12).
$$
At $h=(1,-9,-6,-2)$, the Jacobian determinant of
$(T_2,T_6,T_8,T_{12})$ is congruent to
$
235133\pmod {1000003}.
$
Thus these invariants are algebraically independent. We take
$$
 Q=T_2,\qquad P_6=T_6,\qquad P_8=T_8,\qquad P_{12}=T_{12}.
$$
For $1\leq t\leq10$, set
$$
 x(t)=\sum_{j=1}^{52}c_j(t)x_j,\qquad
 c_j(t)=\bigl(17j^2+31tj+7t^2\bmod11\bigr)-5.
$$
The following matrices are submatrices of
$\bigl(\delta_d(F_j)(x(t))\bigr)$, reduced modulo
$p=1000003$:
\begin{equation}
\begin{array}{c|c|c}
d&\text{matrix}&\det\\
\hline
6&(783144)&783144\\[1mm]
8&\begin{pmatrix}799068&695043\\550969&709965\end{pmatrix}
&612867\\[3mm]
10&
\begin{pmatrix}
122447&384412&414793\\
238126&803313&383695\\
891696&371658&947093
\end{pmatrix}
&994939\\[5mm]
12&
\begin{pmatrix}
337139&712461&551305&573583&664408\\
714074&144322&665362&688898&419163\\
154701&59050&868246&965890&947411\\
454991&444447&184227&47397&561961\\
895572&585010&286906&287616&967123
\end{pmatrix}
&208521.
\end{array}
\label{eq:C:apolar-minors}
\end{equation}
The columns are taken in the order used in
Proposition~\ref{prop:ranks}. The selected rows are the first
$1,2,3,5$ evaluation points, respectively.
Let $F_6\in(I_W^\g)_6$ be the contraction of two copies of $W$.
For the coefficients of $F_6$ and $Q^3$ at the monomials
$h_4^6$ and $h_3^2h_4^4$, the cross determinant is
$-7/{8162933760}\neq0.$  Thus $F_6$ is not proportional to $Q^3$.
In degree eight, the first two evaluation rows on
$(Q^4,QP_6,P_8)$ are
$$
\begin{pmatrix}
799068&695043&579902\\
550969&709965&734184
\end{pmatrix}.
$$
Their kernel is generated by 
\begin{equation}
 (37328,50453,1)\pmod p.
\label{eq:C2}
\end{equation}
Weyl integration reduces the invariant apolar pairing with $Q^4$
to the moment row
$(964592,549089,347104),$ whose pairing with the displayed kernel vector equals
\begin{equation}
 633470\pmod p.
\label{eq:C:apolar-pairing}
\end{equation}
Equations \eqref{eq:C:apolar-minors}--\eqref{eq:C:apolar-pairing} are nonzero integral or rational
certificates. All denominators are invertible modulo $p$.

\section{Poisson-coinvariant proofs of the MLDE order bounds}
\label{app:B3G2}
We summarize the order arguments for $V_B=L_{-2}(B_3)$ and
$V_G=L_{-2}(G_2)$.  For either algebra, put $R=V/C_2(V)$ and
$\mathcal M=\C[E_4,E_6]$.  Denote the filtered Zhu quotient
of Section~\ref{sec:mlde} by $\mathcal Q$, which has a surjection
$$
 \mathcal M\otimes\HP(R)\twoheadrightarrow\gr\mathcal Q.
$$
Put $u_j=L[-2]^j\vac$. For every ordinary $V$-module $M$,
Zhu's recursion formula gives
$Z_M(u_j)=\D^{(j)}\chi_M+
\sum_{i<j}a_{ji}\D^{(i)}\chi_M$, where
$a_{ji}\in\mathcal M_{2j-2i}$
\cite[Proposition~5.2 and Lemmas~5.4--5.6]{ArakawaKawasetsu}.
Thus a homogeneous relation with coefficient one on $u_n$ gives a
common monic order-$n$ equation, without inverting modular forms.

\subsection*{The vertex algebra $L_{-2}(B_3)$}
Let $W=\C^7$, identify $\mathfrak{so}(W)$ with $\bigwedge^2W$, and set
$$
 Y=\{u\wedge v:u,v\in W\},\qquad A=\C[Y],\qquad
 C(X)=-\tfrac12\tr(X^2).
$$
The quadratic singular-vector module is the space
$\bigwedge^4W\cong L_{B_3}(2\omega_3)$ of Pl\"ucker relations,
hence $A\twoheadrightarrow R_{V_B}$.  The degree-six subsingular
relation supplies a nonzero
$U_6\cong L_{B_3}(4\omega_1)\subset A_6$ in its kernel $J$
\cite[Section~5]{ADFLM}.  Nonvanishing follows from the leading
Harish--Chandra polynomial $p_4^{B_3}$ of
\cite[Appendix~B]{ADFLM}:
$(p_4^{B_3})_{[6]}(t\omega_1)=-384t^6$. Here the subscript $[6]$
denotes the homogeneous component of degree six, and the line
$t\omega_1$ lies in $Y$,
so this symbol survives the Pl\"ucker quotient.

The dense locus $C\ne0$ consists of scalar multiples of the orbit
of $X_0=e_1\wedge e_2$.  Its stabilizer is
$K=\mathrm{SO}_2\times\mathrm{SO}_5$, and $X_0$ is conjugate to $-X_0$;
therefore $A^{\mathfrak{so}(W)}=\C[C]$.
Let $\mathcal H_4$ be the harmonic quartics on $W$.
A $K$-invariant quartic $as^2+bst+ct^2$, with
$s=z_1^2+z_2^2$ and $t=z_3^2+\cdots+z_7^2$, is harmonic precisely when
$16a+10b=4b+28c=0$.  Hence $\dim\mathcal H_4^K=1$.
With $Q=z^{\mathsf T}(-X^2)z$ and $S=(z,z)$, define
$$
 H(X;z)=Q^2-\frac8{11}CQS+\frac8{99}C^2S^2.
$$
This is a nonzero degree-four covariant into $\mathcal H_4$.
Density and $\dim\mathcal H_4^K=1$ imply that the degree-six
covariant corresponding to $U_6$ is $aCH$, with $a\ne0$.
The invariant Fischer pairing gives
$\langle H,H\rangle=\frac{2240}{99}C^4$, consequently
$\langle aCH,H\rangle\in J$ implies $C^5\in J$.
The image $[\omega]\in R_{V_B}$ of the Sugawara conformal
vector is a nonzero scalar multiple of the image of $C$, so
$$
 [\omega]^5=0,\qquad
 \HP(R_{V_B})=\operatorname{span}_{\C}
 \{1,[\omega],[\omega]^2,[\omega]^3,[\omega]^4\}.
$$
For the spanning assertion, use exactness of invariants degreewise
and $\mathfrak{so}(W)R_{V_B}\subset\{R_{V_B},R_{V_B}\}$.
Filtered induction now makes $u_0,\ldots,u_4$ generate $\mathcal Q$.
The weight-ten relation for $u_5$ gives an operator
$\D^{(5)}+aE_4\D^{(3)}+bE_6\D^{(2)}
+dE_4^2\D+eE_4E_6$.
To identify it, use
$$
 F=\chi_{V_D}=q^{7/12}(1+28q+329q^2+\cdots),\qquad
 {\chi_{L^B_{-2}(\omega_1)}}
       =q^{19/12}(7+112q+\cdots).
$$
The second coefficient follows from
$\dim(V_B)_2=21+\binom{22}{2}-35=217$ and
$[q^{7/12+2}]F=329$.
After dividing the two non-vacuum coefficient equations by seven,
the coefficient matrix has determinant $1774080$. Its unique solution is
$$
 (a,b,d,e)=\left(-\frac{667}{576},\frac{335}{2304},
                    \frac{8365}{82944},-\frac{23275}{995328}\right).
$$
The resulting operator is exactly $\mathcal P_B=\mathcal L_3\circ
\mathcal L_2$ in \eqref{eq:6:b3-mlde}. Its existence has thus been
established independently of a finite Fourier fit.

\subsection*{The vertex algebra $L_{-2}(G_2)$}
Set $\mathfrak g=G_2$ and work first over its rational Chevalley form.
The defining singular
symbol gives
$$
 R_{V_G}=S(\mathfrak g)/(U),\qquad
 U=U(\mathfrak g)P\cong L_{G_2}(4\omega_1)\subset S^6(\mathfrak g),
 \qquad\dim U=182,
$$
where $P$ is normalized by the descendant identities in
\cite[Proposition~4.1]{ADFLM}.
On $\mathfrak g$ put
$u(X)=\frac18\tr((\ad X)^2)$ and
$v(X)=\frac14\tr((\ad X)^6)$, these generate
$S(\mathfrak g^*)^{\mathfrak g}=\mathbb Q[u,v]$.
The canonical apolar pairing gives
$$
 \dim(R_{V_G})_d^{\mathfrak g}
 =\dim\ker\bigl(P(\partial):
 S^d(\mathfrak g^*)^{\mathfrak g}\longrightarrow
 S^{d-6}(\mathfrak g^*)\bigr).
$$
In the invariant basis $u^{(d-6j)/2}v^j$, the finite evaluation
certificate has the following nonzero square minors modulo $101$:
$$
\begin{array}{c|rrrr}
 d&10&12&14&18\\ \hline
 \text{minor size}&1&2&3&4\\
 \text{minor residue}&79&62&55&23
\end{array}
$$
Thus the invariant dimensions in degrees $10,12$ are at most one,
and those in degrees $14,18$ vanish.
Since $(R_{V_G})^{\mathfrak g}$ is a quotient of
$\mathbb Q[c_2,c_6]$, the latter vanishings kill
$c_2^7,c_2^4c_6,c_2c_6^2,c_6^3$, and hence every invariant monomial
of degree at least $14$.
As above, Poisson coinvariants are a further quotient, so
$$
\begin{array}{c|rrrrrrr}
 d&0&2&4&6&8&10&12\\ \hline
 \dim\HP(R_{V_G})_d\text{ at most}&1&1&1&2&2&1&1
\end{array}
$$
and all other degrees vanish.  Consequently
$\dim\mathcal Q_{22}\le
\sum_d\dim\HP(R_{V_G})_d\dim\mathcal M_{22-d}\le15$.
The extra degree-six invariant means that nilpotence of $[\omega]$
alone does not supply the required scalar order bound.

To obtain a monic relation, consider the fifteen operators
$$
 \mathcal B=\{E_4^aE_6^b\D^{(j)}:
       4a+6b+2j=22,\ 0\le j\le9\}.
$$
Put $B=\chi_{V_B}$, $G=\chi_{V_G}$ and $T=2B-F$.
By \eqref{eq:6:branching-characters}, $F,T,G$ are linear combinations of
ordinary $V_G$-characters.  Their needed coefficients after removing
$q^{7/12}$ are
$$
\begin{array}{c|rrrrrrrr}
 &q^0&q^1&q^2&q^3&q^4&q^5&q^6&q^7\\ \hline
 F&1&28&329&2632&&&&\\
 T&1&14&105&574&2576&10052&&\\
 G&1&14&119&770&4165&19754&84420&z
\end{array}
$$
Here $z=\dim(V_G)_7\in\mathbb Z$ is left unspecified.
The $F,T$ entries follow from the already proved $D_4,B_3$ equations.
The $G$ entries follow from
$q^{-7/12}G=\prod_{m\ge1}(1-q^m)^{-14}-182q^6+O(q^7)$.
Order $\mathcal B$ lexicographically by $(j,a,b)$ and form
$\mathsf N(z)$ using the coefficients of $\mathcal P F$ at $n=0,\ldots,3$,
$\mathcal P T$ at $n=1,\ldots,5$, and $\mathcal P G$ at $n=2,\ldots,7$, for each
column $\mathcal P\in\mathcal B$, where $n$ denotes the exponent $7/12+n$.
Exact elimination gives
$
 \det\mathsf N(z)\equiv9\pmod{59}
 .
$
Only the last row depends on $z$, and the constant and linear
coefficients of the determinant are $9$ and $0$ modulo $59$.
This proves independence of the fifteen trace vectors.
The triangular Zhu identities therefore make the corresponding
$E_4^aE_6^b u_j$ a basis of $\mathcal Q_{22}$.
Expressing $u_{11}$ in that basis proves a common relation
$$
 \bigg(\D^{(11)}+\sum_{j=0}^{9}
 f_j\D^{(j)}\bigg)\chi_M=0,
 \qquad f_j\in\mathbb Q[E_4,E_6]_{22-2j}.
$$
There is no tenth-derivative term since $\mathcal M_2=0$.
This certificate proves existence of a common monic order-eleven
equation, without asserting minimality.  Identification with the
specific factorized operator in \eqref{eq:6:g2-mlde} additionally uses
the equation $\mathcal P_Gf_G=0$ in \eqref{eq:6:vacuum-identities}. The branching argument in Section~\ref{sec:MLDE2} then gives its validity for every ordinary character.

\section*{Acknowledgements}    
The authors acknowledge the use of OpenAI’s ChatGPT as an auxiliary tool during the preparation of this manuscript, including for mathematical discussion and editorial assistance. The authors take full responsibility for the final manuscript.

D.A. was partially supported by the Croatian Science Foundation under the project IP-2022-10-9006  and by the project “Implementation of cutting-edge research and its application as part of the Scientific Center of
Excellence for Quantum and Complex Systems, and Representations of Lie Algebras”, Grant No. PK.1.1.10.0004, co-financed by the European Union through the
European Regional Development Fund - Competitiveness and Cohesion Programme 2021-2027. Part of the research reported in this paper was conducted during D.A.’s visit to USTC, Hefei, in July 2026. He is grateful to USTC for its warm hospitality. 
X.D. would like to thank Kazuya Kawasetsu for helpful discussions. 
K.S. is supported by China NSF grant No.1250010384. K.S. would like to thank Amir Kashani-Poor for useful discussions.

\bigskip


\begin{thebibliography}{99}


\bibitem{Ad}
D.~Adamovi\'c,
Representations of vertex algebras associated to symplectic affine
Lie algebra at half-integer levels,
Ph.D. thesis, University of Zagreb, 1996 (in Croatian).

\bibitem{WExtensions}
D.~Adamovi\'c, T.~Arakawa, T.~Creutzig, A.~R.~Linshaw, A.~Moreau,
P.~M\"oseneder Frajria and P.~Papi,
\emph{$W$-algebras as conformal extensions of affine VOAs},  Advances in Mathematics 503 (2026), 111266.

\bibitem{AM}
D.~Adamovi\'c and A.~Milas,
Vertex operator algebras associated to modular invariant
representations for $A_1^{(1)}$,
Math. Res. Lett. \textbf{2} (1995), no.~5, 563--575.

\bibitem{AKMPP}
D.~Adamovi\'c, V.~G.~Kac, P.~M\"oseneder Frajria, P.~Papi and O.~Per\v se,
\emph{An application of collapsing levels to the representation theory of affine vertex algebras},
Int. Math. Res. Not. IMRN \textbf{2020} (2020), no.~13, 4103--4143.


\bibitem{APNonrational}
D.~Adamovi\'c and O.~Per\v se,
\emph{Representations of certain non-rational vertex operator algebras of affine type},
J. Algebra \textbf{319} (2008), 2434--2450.

\bibitem{AdamovicPerse}
D.~Adamovi\'c and O.~Per\v se,
\emph{Some general results on conformal embeddings of affine vertex operator algebras},
Algebr. Represent. Theory \textbf{16} (2013), 51--64.

\bibitem{AdamovicVukorepa}
D.~Adamovi\'c and I.~Vukorepa,
\emph{A new quasi-lisse affine vertex algebra of type $D_4$},
Commun. Contemp. Math. (2026), 2650062.



\bibitem{ArakawaDS}
T.~Arakawa,
\emph{Associated varieties of modules over Kac--Moody algebras and $C_2$-cofiniteness of $W$-algebras},
Int. Math. Res. Not. IMRN \textbf{2015} (2015), no.~22, 11605--11666.

\bibitem{ArakawaAdmissible}
T.~Arakawa,
\emph{Rationality of admissible affine vertex algebras in the category $\mathcal O$},
Duke Math. J. \textbf{165} (2016), no.~1, 67--93.

\bibitem{ADFLM}
T.~Arakawa, X.~Dai, J.~Fasquel, B.~Li and A.~Moreau,
\emph{On some simple orbifold affine VOAs at non-admissible level arising from rank one 4d SCFTs},
Commun. Math. Phys. (2025) 406:30.



\bibitem{ArakawaVanEkerenMoreau}
T.~Arakawa, J.~van Ekeren and A.~Moreau,
\emph{Singularities of nilpotent Slodowy slices and collapsing
levels of $W$-algebras},
Forum Math. Sigma \textbf{12} (2024), e95.



\bibitem{ArakawaKawasetsu}
T.~Arakawa and K.~Kawasetsu,
\emph{Quasi-lisse vertex algebras and modular linear differential equations},
in: Lie Groups, Geometry, and Representation Theory,
Progr. Math. \textbf{326}, Birkh\"auser/Springer (2018), 41--57.

\bibitem{ArakawaMoreau}
T.~Arakawa and A.~Moreau,
\emph{Joseph ideals and lisse minimal $W$-algebras},
J. Inst. Math. Jussieu \textbf{17} (2018), 397--417.

\bibitem{ArakawaMoreauSheets}
T.~Arakawa and A.~Moreau,
\emph{Sheets and associated varieties of affine vertex algebras},
Adv. Math. \textbf{320} (2017), 157--209;
corrigendum, Adv. Math. \textbf{372} (2020), 107302.

\bibitem{AxtellLee}
J.~D.~Axtell and K.-H.~Lee,
\emph{Vertex operator algebras associated to type $G$ affine Lie algebras},
J. Algebra \textbf{337} (2011), no.~1, 195--223.

\bibitem{BeemEtAl}
C.~Beem, M.~Lemos, P.~Liendo, W.~Peelaers, L.~Rastelli and B.~C.~van Rees,
\emph{Infinite chiral symmetry in four dimensions},
Comm. Math. Phys. \textbf{336} (2015), no.~3, 1359--1433.

\bibitem{BeemRastelli}
C.~Beem and L.~Rastelli,
\emph{Vertex operator algebras, Higgs branches, and modular differential equations},
J. High Energy Phys. \textbf{08} (2018), 114.

\bibitem{CollingwoodMcGovern}
D.~H.~Collingwood and W.~M.~McGovern,
\emph{Nilpotent Orbits in Semisimple Lie Algebras},
Van Nostrand Reinhold, New York (1993).

\bibitem{Cvitanovic:1980ed}
P.~Cvitanovi\' c,
\emph{Negative Dimensions and the Emergence of $E_7$ Symmetry in Supergravity},
Nucl. Phys. B \textbf{188}, 373-396 (1981).

\bibitem{Deligne}
P.~Deligne.
\emph{La s\'{e}rie exceptionnelle de groupes de {L}ie}, 
C. R. Acad. Sci. Paris S\'{e}r. I Math., 322\penalty0 (4):\penalty0 321--326, 1996.

\bibitem{Duque:2022tub}
D.~J.~Duque and A.~K.~Kashani-Poor,
\emph{Affine characters at negative level and elliptic genera of non-critical strings},
JHEP \textbf{07}, 208 (2023).


\bibitem{FrenkelZhu}
I.~B.~Frenkel and Y.~Zhu,
\emph{Vertex operator algebras associated to representations of affine and Virasoro algebras},
Duke Math. J. \textbf{66} (1992), no.~1, 123--168.

\bibitem{FJLS23}
 B. Fu, D. Juteau, P. Levy
                   and E. Sommers,
\emph{Local geometry of special pieces of nilpotent orbits},
(2023).

\bibitem{GaberdielKeller}
M.~R.~Gaberdiel and C.~A.~Keller,
\emph{Modular differential equations and null vectors},
J. High Energy Phys. \textbf{09} (2008), 079.

\bibitem{GorelikKac}
M.~Gorelik and V.~Kac,
\emph{On simplicity of vacuum modules},
Adv. Math. \textbf{211} (2007), no.~2, 621--677.

\bibitem{KacRoanWakimoto}
V.~G.~Kac, S.-S.~Roan and M.~Wakimoto,
\emph{Quantum reduction for affine superalgebras},
Comm. Math. Phys. \textbf{241} (2003), 307--342.

\bibitem{LiAbelianizing}
H.~Li,
\emph{Abelianizing vertex algebras},
Comm. Math. Phys. \textbf{259} (2005), no.~2, 391--411.

\bibitem{Li94}
H.~Li,
\textit{Symmetric invariant bilinear forms on vertex operator algebras},
J. Pure Appl. Algebra \textbf{96} (1994), no.~3, 279--297.

\bibitem{LiXieYan}
B.~Li, D.~Xie and W.~Yan,
\emph{On low rank 4d $\mathcal N=2$ SCFTs},
J. High Energy Phys. \textbf{05} (2023), 132.

\bibitem{PerseD}
O.~Per\v se,
\emph{A note on representations of some affine vertex algebras of type $D$},
Glas. Mat. Ser. III \textbf{48(68)} (2013), no.~1, 81--90.

\bibitem{Perse}
O.~Per\v se,
\emph{Vertex operator algebra analogue of embedding of $B_4$ into $F_4$},
J. Pure Appl. Algebra \textbf{211} (2007), 702--720.

\bibitem{Zhu}
Y.~Zhu,
\emph{Modular invariance of characters of vertex operator algebras},
J. Amer. Math. Soc. \textbf{9} (1996), no.~1, 237--302.



\end{thebibliography}
\end{document}